\documentclass[english,a4paper,12pt]{amsart}
\usepackage[a4paper,lmargin=2cm,rmargin=2cm,tmargin=4cm,bmargin=4cm]{geometry}
\usepackage{amsfonts}
\usepackage{amssymb}
\usepackage[colorlinks]{hyperref}
\usepackage{tikz}
\usepackage{tikz-cd}
\usepackage[normalem]{ulem}
\usepackage{tkz-euclide}
\usepackage{bbm}
\usepackage{mathabx}
\usepackage[labelfont=bf]{subcaption}

\usepackage{xspace} 
\usepackage{enumitem} 
\setenumerate[1]{label=(\roman{*}), ref=(\roman{*})}
\setenumerate[2]{label=(\alph{*}), ref=(\alph{*})}
\usepackage{marginnote}

\definecolor{egraf}{rgb}{0.2,0.4,0}
\definecolor{fill1}{rgb}{0.170,0.437,0.224}
\definecolor{fill2}{rgb}{0.633,0.475,0.290}
\definecolor{fill3}{rgb}{0.830,0.563,0.776}
\definecolor{fill4}{rgb}{0.763,0.850,0.951}

\newcommand{\abs}[1]{\left\vert#1\right\vert}
\newcommand{\set}[1]{\left\{#1\right\}}

\newcommand\restr[2]{{
    \left.\kern-\nulldelimiterspace 
      #1 
      \littletaller 
    \right|_{#2} 
  }}
\newcommand{\littletaller}{\mathchoice{\vphantom{\big|}}{}{}{}}

\newcommand{\diam}{\mathop{\mathrm{diam}}\nolimits}

\newcommand{\norm}[1]{\left\Vert#1\right\Vert}
\newcommand{\duality}[1]{\left\langle#1\right\rangle}
\newcommand{\clco}{\mathop{\overline{\mathrm{co}}}\nolimits}

\newcommand{\linsp}{{\mathrm{span}}}

\newcommand{\Lip}{{\mathrm{Lip}}}

\newcommand{\conv}{\mathop\mathrm{conv}}
\newcommand{\supp}{\mathop\mathrm{supp}}
\newcommand{\ext}{\mathop\mathrm{ext}}
\newcommand{\dent}{\mathop\mathrm{dent}}

\newcommand{\eps}{\varepsilon}
\newcommand{\sign}{\mbox{sign}\,}
\newcommand{\cconv}{\overline{\mbox{conv}}}

\newcommand{\lipfree}[1]{\mathcal{F}({#1})}                 
\newcommand{\Ymap}{\mathsf{Y}}

\theoremstyle{plain}
\newtheorem{thm}{Theorem}[section]
\newtheorem{cor}[thm]{Corollary}
\newtheorem{lem}[thm]{Lemma}
\newtheorem{prop}[thm]{Proposition}

\newtheorem{fact}[thm]{Fact}
\newtheorem*{thm*}{Theorem}
\newtheorem*{prop*}{Proposition}

\theoremstyle{definition}
\newtheorem{defn}[thm]{Definition}
\newtheorem{rem}[thm]{Remark}

\newtheorem{example}[thm]{Example}

\newtheorem{obs}[thm]{Observation}
\newtheorem{quest}[thm]{Question}
\newenvironment{claim}[1]{\vskip 1mm\par\noindent\underline{Claim.}\space#1}{}
\newenvironment{claimproof}[1]{\vskip
  1mm\par\noindent\underline{Proof.}\space#1}{\hfill $\blacksquare$
  \vskip 1mm}

\newcommand{\N}{\mathbb N}

\newcommand{\R}{\mathbb R}

\newcommand{\1}{\mathbbm{1}}

\newcommand{\RNP}{Radon--Nikod\'{y}m property\xspace}
\newcommand{\KMP}{Krein--Milman property\xspace}

\newcommand{\nnorm}[1]{\left\vvvert#1\right\vvvert}
\newcommand{\scal}[1]{\left\langle#1\right\rangle}
\newcommand{\spn}{\mathop{\mathrm{span}}\nolimits}
\newcommand{\cspn}{\overline{\mbox{span}}}

\author[T.~A.~Abrahamsen]{Trond A.~Abrahamsen}
\address[T.~A.~Abrahamsen]{Department of Mathematics, University of
  Agder, Postboks 422, 4604 Kristiansand, Norway.}
\email{trond.a.abrahamsen@uia.no}
\urladdr{http://home.uia.no/trondaa/index.php3}

\author[R.~J.~Aliaga]{Ram\'on J. Aliaga}
\address[R.~J.~Aliaga]{Instituto Universitario de Matem\'atica Pura y Aplicada,
Universitat Polit\`ecnica de Val\`encia,
Camino de Vera S/N,
46022 Valencia, Spain}
\email{raalva@upv.es}

\author[V.~Lima]{Vegard Lima}
\address[V.~Lima]{Department of Engineering Sciences, University of Agder,
Postboks 509, 4898 Grimstad, Norway.}
\email{vegard.lima@uia.no}

\author[A.~Martiny]{Andr\'e Martiny}
\address[A.~Martiny]{Department of Mathematics, University of
  Agder, Postboks 422, 4604 Kristiansand, Norway.}
\email{andre.martiny@uia.no}

\author[Y.~Perreau]{Yoël Perreau}
\address[Y.~Perreau]{University of Tartu, Institute of Mathematics and Statistics, Narva mnt 18, 51009 Tartu
  linn, Estonia}
\email{yoel.perreau@ut.ee}

\author[A.~Prochazka]{Anton\'in Prochazka}
\address[A.~Prochazka]{Université de Franche-Comté,
Laboratoire de mathématiques de Besançon,
UMR CNRS 6623,
16 route de Gray,
25000 Besançon,
France}
\email{antonin.prochazka@univ-fcomte.fr}

\author[T.~Veeorg]{Triinu Veeorg}
\address[T.~Veeorg]{University of Tartu, Institute of Mathematics and Statistics, Narva mnt 18, 51009 Tartu
  linn, Estonia}
\email{triinu.veeorg@ut.ee}

\title{A relative version of Daugavet-points and the Daugavet property}

\begin{document}

\begin{abstract}
  We introduce relative versions of Daugavet-points and the Daugavet
  property, where the Daugavet-behavior is localized inside of some
  supporting slice.
  These points present striking similarities with Daugavet-points,
  but lie strictly between the notions of Daugavet- and
  $\Delta$-points.
  We provide a geometric condition that a space with the
  Radon--Nikod\'{y}m property must
  satisfy in order to be able to contain a relative Daugavet-point.
  We study relative Daugavet-points in absolute sums of Banach
  spaces, and obtain positive stability results under local
  polyhedrality of the underlying absolute norm.
  We also get extreme differences between the relative Daugavet
  property, the Daugavet property, and the diametral local diameter 2
  property.
  Finally, we study Daugavet- and $\Delta$-points in subspaces of
  $L_1(\mu)$-spaces.
  We show that the two notions coincide in the class of all
  Lipschitz-free spaces over subsets of $\mathbb{R}$-trees.
  We prove that the diametral local diameter 2 property and the
  Daugavet property coincide for arbitrary subspaces of $L_1(\mu)$,
  and that reflexive subspaces of $L_1(\mu)$ do not contain
  $\Delta$-points.
  A subspace of $L_1[0,1]$ with a large subset of
  $\Delta$-points, but with no relative Daugavet-point,
  is constructed.
\end{abstract}
\maketitle
\markleft{\textsc{Abrahamsen et al.}}

\section{Introduction}
\label{sec:intro}

Since their introduction in \cite{AHLP}, Daugavet- and $\Delta$-points
have been intensively studied in classical and semi-classical Banach
spaces \cite{AHLP, ALMT, ALM, JungRueda, VeeorgStudia}, and strong
efforts have been put into understanding the influence of those points
for the geometry of the target spaces \cite{ALMP, VeeorgFunc, KLT,
  MPRZ, AALMPPV, HLPV23}. Recall that a point $x$ on the unit sphere
of a Banach space $X$ is a \emph{Daugavet-point} (respectively a
\emph{$\Delta$-point}) if we can recover, for any given $\eps>0$, the
closed unit ball $B_X$ of the space $X$ (respectively the point $x$
itself) by taking the
closed convex hull of the set $\Delta_\eps(x)$ of those points of
$B_X$ which are at distance greater than $2-\eps$ to the point $x$.
Using the Hahn--Banach theorem the above definitions can
be reformulated in terms of slices of the unit ball
(See Section~\ref{sec:preliminaries}).

In view of their definitions, it was first expected that Daugavet- and
$\Delta$-points were to have very strong implications for the
geometry of Banach spaces. However, striking examples, such
as a Lipschitz-free Banach space with the \RNP and a Daugavet-point
\cite{VeeorgStudia}, or a Banach space with a 1-unconditional basis
and a large subset of Daugavet-points \cite{ALMT}, have shown that the
situation is much more complex than anticipated. Very recently, the
authors of the present note have also definitely established that
the notion of $\Delta$-points is purely isometric by proving that every
infinite dimensional Banach space can be renormed in such a way that
it admits a $\Delta$-point \cite[Theorem~4.5]{AALMPPV}. Similar results 
were obtained for Daugavet-points in \cite{HLPV23} for spaces with 
an unconditional basis.

In that same paper, it was pointed out that most of the results
which are currently known to provide an obstruction to the existence
of Daugavet-points in Banach spaces are actually coming from an
obstruction to
$\Delta$-points. In the present note, we aim at giving new insights on
this question, and more generally at getting a better understanding of
Daugavet-points in some specific classes of Banach spaces, namely
spaces with the \RNP or the \KMP, and subspaces of $L_1[0,1]$.

The study of Daugavet-points in this context naturally leads to consider 
a new notion of points on the unit sphere of Banach spaces, which,
roughly speaking, localizes the Daugavet-behavior inside of some slice
of the unit ball that is defined by a supporting functional. We
introduce these relative
Daugavet-points in Section~\ref{subsec:relative_Daugavet-points}, and
show that the notion naturally encompasses Daugavet-points. Next, we
establish that those points present a very similar relationship to
neighboring denting points as Daugavet-points do, and that the notion
lies strictly between the notions of Daugavet- and
$\Delta$-points. Finally, we  prove that every Banach space with a
weakly null basis (and more generally with a weakly null complemented
basic sequence) can be renormed with a relative Daugavet-point.

In Section~\ref{subsec:convex_representations}, we build on this
property of distance 2 to denting points 
that was first observed in \cite{JungRueda} to provide new results on
(relative) Daugavet-points. This allows us to
provide a natural isometric condition that guarantees that a
given space with the \RNP fails to contain relative Daugavet-points. More
precisely, we show that if all the extreme points of the unit ball of
a Banach space $X$ with the \KMP are denting, then this space cannot
contain a relative Daugavet-point. In fact, we show under some additional
assumption that if a space with the \KMP contains a relative
Daugavet-point, then it must contain a point that is simultaneously a
relative Daugavet-point and an extreme point of the unit ball.

In Subsection~\ref{subsec:relative_Daugavet-property}
we investigate the property that every point on the unit sphere of a
given Banach space is a relative Daugavet-point. Establishing specific
transfer results for these points under the operation of taking
absolute sums of Banach spaces, we show that this relative Daugavet
property lies
strictly between the Daugavet property and the diametral local diameter 2
property (DLD2P, see Section~\ref{sec:preliminaries}). 
To be more precise, we characterize, given a point $(a,b)$ in the
positive cone of $\R^2$ and two relative Daugavet-points $x\in S_X$
and $y\in S_Y$, the property that $(ax,by)$
is a relative Daugavet-point in the absolute sum $X\oplus_N Y$ by a
specific geometric condition on $(a,b)$ which corresponds to a local
polyhedrality of the unit ball of $(\R^2,N)$ around this point.
As an application, we construct a Banach space with the relative 
Daugavet property that does not contain a Daugavet-point, 
and a Banach space with the DLD2P that does not contain a relative
Daugavet-point.

Next, we study the behavior of (relative) Daugavet-points in the
class of all subspaces of $L_1[0,1]$. To motivate this study, let us
recall that it was shown in \cite{JungRueda} that Daugavet- and
$\Delta$-points naturally present strong differences in Lipschitz-free
spaces. In the opposite direction, it
was proved in \cite{AHLP} that the two notions  always coincide in
$L_1(\mu)$ spaces. As the
Lipschitz-free spaces which embed isometrically into some $L_1(\mu)$
space (respectively into $L_1[0,1]$)
are precisely those over (separable) subsets of $\R$-trees, and
as the extreme points in the unit ball of such spaces are all denting,
it makes sense to study in
greater detail the behavior of Daugavet- and $\Delta$-points in this
setting.  We will prove
in Subsection \ref{subsec:subL1-free} that in the latter setting
Daugavet- and $\Delta$-points are the same, so that Theorem
\ref{thm:KMP_denting} also allows to exclude $\Delta$-points in such
spaces in presence of the \KMP. In particular, it
naturally opens up the question whether Daugavet- and
$\Delta$-points are the same in all subspaces of $L_1[0,1]$. 
In Subsection~\ref{subsec:subL1-reflexive} we show that this
question is trivial for reflexive subspaces $L_1[0,1]$
since those do not contain $\Delta$-points (in fact they do not even
contain $\mathfrak{D}$-points by Proposition~\ref{prop:4}). We will
also prove that the diametral local diameter 2 property and the
Daugavet property coincide for general subspaces of $L_1[0,1]$ (in
fact these properties are even equivalent to property
$(\mathfrak{D})$ in this context, see
Theorem~\ref{thm:propD-implies-Daugprop}).
In contrast, we will construct in
Subsection~\ref{subsec:subL1_relative-Daugavet_no_daugavet} a subspace
of $L_1[0,1]$ with a large subset of $\Delta$-points, but that
contains no relative Daugavet-point.


\section{Notation and preliminary results}\label{sec:preliminaries}

We will use standard notation, and will usually follow the textbooks
\cite{AlbiacKalton} or \cite{MR2766381}. For a given Banach space
$X$, we denote respectively by $B_X$ and $S_X$ the closed unit ball
and the unit sphere of $X$, and we denote by $X^*$ the topological
dual of $X$. Also, for a given subset $A$ of $X$, we denote
respectively by $\conv(A)$, $\cconv(A)$, $\spn(A)$ and $\cspn(A)$ the
convex hull, closed convex hull, linear span, and closed linear span
of $A$. We will only consider real Banach spaces.

Let $X$ be a Banach space and $C$ be a non-empty, bounded and
convex subset of $X$. Recall that a point $x\in C$ is called an
\emph{extreme point} of $C$ if it does not belong to the interior of
any non-trivial segment of $C$. We will denote by $\ext{C}$ the set of
all extreme points of $C$. Also recall that a non-empty convex subset
$F$ of $C$ is called a \emph{face} (or an \emph{extremal subset}) of
$C$ if every non-trivial segment of $C$ whose interior has non-empty
intersection with $F$ is fully contained in $F$. It is well known that
$\ext{F}=F\cap \ext{C}$ for any given face $F$ of $C$.

We will denote by 
$$S(C,x^*,\alpha):=\left\{x\in B_X:\ x^*(x)>\sup x^*(C)-\alpha\right\}
$$
the \emph{slice} of $C$ defined by the
functional $x^*\in X^*$ and the positive real number $\alpha$. When
$C$ is equal to $B_X$, we will simply write
$S(x^*,\alpha):=S(B_X,x^*,\alpha)$. 
Recall that a
point $x\in C$ is called a \emph{denting point} of $C$ if it is
contained in slices of $C$ of arbitrarily small diameter. We will
denote by $\dent{C}$ the set of all denting points of $C$. By the
classic result from Lin, Lin and Troyanski \cite{LLT}, we
have that the point $x$ is denting if and only if it is simultaneously
an extreme point of $C$ and a point of \emph{weak-to-norm continuity}
(\emph{point of continuity} in short), i.e. a point of continuity for
the identity map $I:(C,w)\to (C,\norm{\cdot})$.

In the first part of the present note, we will consider spaces with 
the \RNP (RNP) and the \KMP (KMP). We refer e.g. to \cite{Bourgin} for
background on the topic, and in particular for classic geometric
characterizations of these properties. Recall the RNP implies the KMP,
and that it is still an open question in general whether the KMP and
the RNP are equivalent. However, note that this is known to be true
for dual spaces from the result by Huff and Morris
\cite{HM}, as well as for Lipschitz-free spaces from the recent result
by Aliaga, Gartland, Petitjean and Proch\'{a}zka \cite{AGPP}.
At the extreme opposite, recall that $X$ has the \emph{local diameter 2
property} (\emph{LD2P}) if every slice of $B_X$ has maximal
diameter in the ball, that is diameter 2. 

We will denote by $D(x)$ the set of all
norm-one supporting functionals at a given point $x$ in the unit sphere of
a Banach space $X$, i.e.
\begin{equation*}
  D(x):=\left\{x^*\in S_{X^*}:\ x^*(x)=1\right\}.
\end{equation*}
We will also call a slice $S$ a \emph{supporting slice} at the point
$x \in S_X$ if $S$ is a slice of $B_X$ defined by a supporting functional
$x^*\in D(x)$.
The following points were first studied in \cite{AHLP}.

\begin{defn}\label{defn:slice_diametral_points}
  Let $X$ be a Banach space.
  We say that $x \in S_X$ is a
  \begin{enumerate}
  \item \emph{Daugavet-point} if
    $\sup_{y \in S}\norm{x-y}=2$ for every slice $S$ of $B_X$;
  \item \emph{$\Delta$-point} if
    $\sup_{y \in S}\norm{x-y}=2$ for every slice $S$ of $B_X$
    containing $x$;
  \item \emph{$\mathfrak{D}$-point}
    if $\sup_{y \in S}\norm{x-y}=2$ for every supporting slice
    $S$ at $x$.
  \end{enumerate}
\end{defn}

Recall that Daugavet-, $\Delta$- and $\mathfrak{D}$-points were
introduced as respective ``localizations''  of the Daugavet
property, of the diametral local diameter 2 property, and of
property $(\mathfrak{D})$.
This means that a Banach space $X$ has the Daugavet property
if and only if every $x \in S_X$ is a Daugavet-point.
Similarly $X$ has the \emph{diametral local diameter 2 property}
(DLD2P in short) if and only if every $x \in S_X$ is a $\Delta$-point,
and $X$ has property $(\mathfrak{D})$
if and only if every $x \in S_X$ is a $\mathfrak{D}$-point.
Note that the DLD2P is also known as ``space with bad projections''
\cite{zbMATH02168839}.
It is clear that the Daugavet property implies the DLD2P
which in turn implies property $(\mathfrak{D})$.
As shown in \cite[Proposition~2.5]{AHLP} property $(\mathfrak{D})$
implies the LD2P.

Let us also recall right away the following
principle of ``localization of slices'',
\cite[Lemma~2.1]{zbMATH02168839}, that is extremely useful for the
study of Daugavet-, $\Delta$- and $\mathfrak{D}$-points.

\begin{lem}\label{lem:diminution_slices}
  Let $X$ be a Banach space, $x^*\in S_{X^*}$ and
  $\alpha>0$. For every $x \in S(x^*,\alpha) \cap S_X$ and every
  $\beta\in(0,\alpha)$ there exists $y^*\in S_{X^*}$ such that
  \begin{equation*}
    x\in S(y^*,\beta)\subseteq S(x^*,\alpha).
  \end{equation*}
\end{lem}

We will also consider stronger versions of Daugavet- and
$\Delta$-points which were recently introduced in \cite{MPRZ}.

\begin{defn}\label{defn:super_points}

Let $X$ be a Banach space and $x\in S_X$. We say that $x$ is a
\begin{enumerate}

    \item \emph{super Daugavet-point} if $\sup_{y\in
        W}\norm{x-y}=2$ for every non-empty relatively weakly open
      subset $W$ of $B_X$;

    \item \emph{super $\Delta$-point} if $\sup_{y\in
        W}\norm{x-y}=2$ for every relatively weakly open subset $W$ of
      $B_X$ containing $x$.
      
\end{enumerate}

\end{defn}

By a \emph{relatively weakly open subset} of $B_X$ we mean, as usual,
any subset $W$ of $B_X$ which is obtained by intersecting $B_X$ with a
weakly open subset $V$ of $X$. Other variations involving \emph{convex
  combinations of slices} (or of relatively weakly open subsets) were
also introduced in \cite{MPRZ}, but we will, except for a few side
remarks or questions, not consider them in the present note.

From the geometric characterizations from \cite{KSSW}, super
Daugavet-points are also a natural localization of the
Daugavet property. On other hand, super $\Delta$-points are a
localization of the \emph{diametral diameter 2 property} (\emph{DD2P})
introduced in \cite{BGLPRZ_DD2Ps}. Note that it is currently
unknown whether the DD2P and the DLD2P are equivalent.

From the definitions, and from the fact that the slices we consider
are relatively weakly open, we clearly have the following diagram.

\bigbreak

\begin{figure}[hbt!]
\begin{tikzcd}
\fbox{\text{super Daugavet}} \arrow [d, Rightarrow] \arrow[r, Rightarrow] &
\fbox{\text{super $\Delta$}} \arrow [d, Rightarrow] \\
\fbox{\text{Daugavet}} \arrow[r, Rightarrow] & \fbox{\text{$\Delta$}} \arrow [r,
Rightarrow] & \fbox{\text{$\mathfrak{D}$}}
\end{tikzcd}
\caption{Relations between the Daugavet- and
  $\Delta$-notions}\label{figure:relations_diametral_notions}
\end{figure}
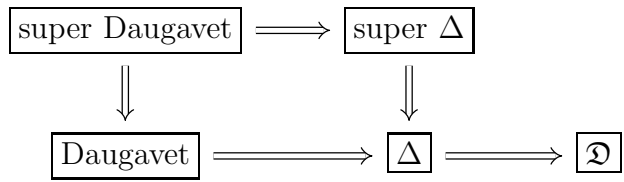

\bigbreak

However, none of the previous implications reverses. In fact it turns
out that, apart for very few specific classes of Banach spaces, all
those notions behave differently, and might even present extreme
differences. For example, let us point out that,
combining \cite[Theorem~2.11]{BGLPRZ_DD2Ps} and
\cite[Proposition~4.6]{AHLP}, $C[0,1]\oplus_2 C[0,1]$
has the DD2P but contains no Daugavet-point. Also,
Abrahamsen, Lima, Martiny and Troyanski in \cite{ALMT} gave an
example of a Banach space $X_{\mathfrak{M}}$ which admits a
1-unconditional basis, and which contains a subset $D_B$ of elements
of $S_{X_\mathfrak{M}}$ which are simultaneously Daugavet-points and
points of continuity (so which are contained in relatively weakly open
subset of $B_{X_\mathfrak{M}}$ of arbitrarily small diameter)
such that $B_{X_\mathfrak{M}}=\cconv(D_B)$. Lastly, every element in
$S_{\ell_1}$ with infinite support is a
$\mathfrak{D}$-point by \cite[Proposition~2.3]{AHLP}, but
$\ell_1$ does not contain any $\Delta$-points
\cite[Theorem~3.1]{AHLP}. We refer to \cite{AHLP},
\cite{JungRueda} and \cite{MPRZ} for more examples and for further
discussions. We will also provide new results in this direction for
subspaces $L_1[0,1]$ in Section
\ref{subsec:subL1_relative-Daugavet_no_daugavet}.

\bigbreak

Finally, let us end the present section by recalling a few basic
definitions and facts about \emph{Lipschitz-free spaces}. We refer to
\cite{Weaver2} for a thorough presentation of the topic. Let $(M,d)$
be a \emph{pointed} metric space, that is a metric space $M$ with a
designated origin denoted by $0$. We denote by $\Lip_0(M)$ the space
of real valued Lipschitz functions $f$ on $M$ which vanish at $0$ equipped
with the norm
\[
  \|f\|
  = \sup\left\{ \frac{|f(x) - f(y)|}{d(x,y)}:\ x, y \in M,\ x \neq y \right\}.
\]

Let $\delta: M \to \Lip_0(M)^*$ be the canonical embedding of $M$ into
$\Lip_0(M)^*$ given by $x \mapsto \delta_x$ where $\delta_x(f) =
f(x).$ The space $\lipfree{M} = \overline{\mbox{span}}\,\delta(M)
\subset \Lip_0(M)^*$ is called the \emph{Lipschitz-free space} over
  $M$. It is well known that the map $\delta$ is a non-linear
  isometry, and that $\Lip_0(M)$ is a predual of $\lipfree{M}$. For $x
  \neq y \in M$ we set
\[
  m_{xy}
  := \frac{\delta_x - \delta_y}{d(x,y)},
\]
and call such an element a \emph{molecule}. Clearly, every molecule
belongs to the unit sphere of $\lipfree{M}$, and it is well known that
the unit ball of $\lipfree{M}$ is equal to the closed convex hull of
the set of all molecules over $M$.  Also, for every $f\in \Lip_0(M)$
and for every $x\neq y$ in $M$, we
write $$f(m_{xy}):=\scal{m_{xy},f}=\frac{f(x)-f(y)}{d(x,y)}.$$ More
generally, we write $f(\mu):=\scal{\mu,f}$ for every $\mu\in
\lipfree{M}$.  Recall that a Lipschitz function $f$ is  said to be
$\emph{local}$ if for every $\eps>0$, there exist $u\neq v\in M$ such
that $d(u,v)<\eps$ and $f(m_{uv})>\norm{f}-\eps$.  Lipschitz-free
spaces as well as spaces of Lipschitz functions have proved to be
natural interesting settings for the study of Daugavet- and
$\Delta$-points, and we refer to \cite{JungRueda},
\cite{VeeorgStudia}, \cite{VeeorgFunc} and \cite{AALMPPV} for results
in this direction.


\section{Relative Daugavet-points and the relative Daugavet property}\label{sec:relative_Daugavet-points}

\subsection{Relative Daugavet-points}\label{subsec:relative_Daugavet-points}

In this section, we introduce the notion of relative Daugavet-points
which, as we shall see, resembles locally very much the
Daugavet-notion, but lies strictly between the notions of Daugavet-
and $\Delta$-points. Using ideas from \cite{HLPV23}, we also prove
that every infinite dimensional Banach space with a weakly null Schauder basis
(and in fact every Banach space with a complemented weakly null basic
sequence) can be renormed with a relative Daugavet-point.

\begin{defn}\label{defn:relative_Daugavet-points}
    Let $X$ be a Banach space, $x\in S_X$, $x^*\in D(x)$ and $\alpha>0$. 
    We say that $x$ is a Daugavet-point \emph{relative to the
    slice $S:=S(x^*,\alpha)$} if $\sup_{y\in T}\norm{x-y}=2$ for every \emph{subslice}
    $T$ of $S$, that is for every slice $T$ of $B_X$ which is contained in $S$. 
    We say that $x$ is a \emph{relative Daugavet-point} if it is a Daugavet-point
    relative to some supporting slice at $x$.
\end{defn}

Clearly, every Daugavet-point is a relative
Daugavet-point. Furthermore, it is worth noticing that the
Daugavet-notion is naturally included in its relative version. Indeed,
we have the following easy lemma.

\begin{lem}\label{lem:Daugavet_implies_relative_Daugavet}
Let $X$ be a Banach space and $x\in S_X$. The following assertions are equivalent.
\begin{enumerate}
    \item The point $x$ is a Daugavet-point.

    \item The point $x$ is a Daugavet-point relative to any supporting slice at $x$.

    \item The point $x$ is a Daugavet-point relative to the slice $S(x^*,2)$ for some supporting functional $x^*\in D(x)$.
\end{enumerate}
\end{lem}

\begin{proof}
    The only non-trivial implication in this statement is that if $x$
    is a Daugavet-point relative to the supporting slice
    $S:=S(x^*,2)$, then $x$ is actually a Daugavet-point. So assume
    that the former holds, and pick an arbitrary slice $T$ of
    $B_X$. If $T$ is contained in $S$, then $\sup_{y\in
      T}\norm{x-y}=2$ by assumption. So let us assume that $T$ is not
    contained in $S$. Then there exists $y\in T$ such that
    $x^*(y)=-1$. As $x^*(x)=1$, we get $\norm{x-y}\geq x^*(x-y)=2$,
    and it follows that $x$ is a Daugavet-point.
\end{proof}

Next, we will show that every relative Daugavet-point is a
$\Delta$-point. For this purpose, let us first prove the following
characterization.

\begin{lem}\label{lem:relative-Daugavet_characterization}
Let $X$ be a Banach space and $x\in S_X$. If $x$ is a
Daugavet-point relative to the supporting slice $S:=S(x^*,\alpha)$, then for
every subslice $T$ of $S$ and $\eps>0$, we can find a
subslice $U$ of $T$ such that $\norm{x-y}>2-\eps$ for every $y\in U$.
\end{lem}

\begin{proof}
  Fix any subslice $T:=S(y^*,\beta)$ of $S$ and $\varepsilon>0$.
  We may assume $2\beta<\varepsilon$.
  Since $S(y^*,\beta/4)\subseteq T\subseteq S$,
  we can find  $y\in S(y^*,\beta/4)$ satisfying
  $\norm{x-y}>2-\beta/4$.
  Let $z^*\in S_{X^*}$ be such that $z^*(y)-z^*(x)>2-\beta/4$.
  Note that
  \[
    \norm{y^*+z^*}\ge
    y^*(y)+z^*(y)>1-\frac{\beta}{4}+1-\frac{\beta}{4}=2-\frac{\beta}{2},
  \]
 and consider 
 \[U:=S\Big(\frac{y^*+z^*}{\norm{y^*+z^*}},\frac{\beta}{2\norm{y^*+z^*}}\Big).\]
 Then for every $z\in U$ we have
 \[y^*(z)+z^*(z)>\norm{y^*+z^*}-\frac{\beta}{2}>2-\beta,\]
 and thus $U\subseteq S(y^*,\beta)\cap S(z^*,\beta) \subseteq T$.
 Furthermore,
 \[
   \norm{x-z}\ge
   z^*(z)-z^*(x)>1-\beta+1-\frac{\beta}{4}>2-\varepsilon,
 \]so we are done.
\end{proof}

We can now prove the following result.

\begin{lem}\label{lem:relative-Daugavet_implies_Delta}
  Every relative Daugavet-point is a
  $\Delta$-point.
\end{lem}

\begin{proof}
  Let us assume that $x \in S_X$ is a Daugavet-point relative to
  the supporting slice $S:=S(x^*,\alpha)$, and let us fix a slice
  $T:=S(y^*,\beta)$ of $B_X$ containing $x$. We will construct a
  subslice $U$ of $T$ that is also contained in $S$. By the principle
  of localization of slices (Lemma \ref{lem:diminution_slices}), we
  can find for every $\gamma\in(0,\beta)$ a subslice
  $T_\gamma:=S(y^*_\gamma,\gamma)$ of $T$ containing the point $x$. So
  pick any positive $\gamma<\min\{\beta,\alpha\}$, and
  define $z^*:=\frac{x^*+y^*_\gamma}{2}$. Then pick any positive
  $\mu<\min\{\gamma,y^*_\gamma(x)-(1-\gamma)\}$, and consider the
  slice $$U:=\left\{y\in B_X:\ z^*(y)>z^*(x)-\frac{\mu}{2}\right\}.$$
  For every $y\in U$, we have
  \begin{align*}
    1 + y^*_\gamma(y) \ge 2z^*(y)
    > 2z^*(x)-\mu
    = x^*(x) + y^*_\gamma(x) - \mu
    = 1 + y^*_\gamma(x) - \mu,
  \end{align*}
  hence
  \begin{equation*}
    y^*_\gamma(y) > y^*_\gamma(x) - \mu > 1 - \gamma,
  \end{equation*}
  which means $U \subset T_\gamma\subset T$.
  Also,
  \begin{equation*}
    x^*(y) + 1 \ge 2z^*(y)
    > x^*(x) + y^*_\gamma(x) - \mu
    > 1 + 1 - \gamma,
  \end{equation*}
  hence
  \begin{equation*}
    x^*(y) > 1 - \gamma > 1-\alpha,
  \end{equation*}
  so that $U\subset S$.
  Consequently, we get
  \begin{equation*}
    \sup_{y\in T}\norm{x-y}\geq \sup_{y\in U}\norm{x-y} =2,
  \end{equation*}
  and the conclusion follows.
\end{proof}

Finally, we will show, as announced, that the relative Daugavet-notion lies strictly
between the Daugavet- and the $\Delta$-notions. In order to provide
such examples, we will need some considerations on distance to denting
points for relative Daugavet-points. First note that the following
result generalizes the observation from \cite{JungRueda}.  We will
investigate this property and its applications in more details in the
following subsection.

\begin{prop}\label{prop:distance-to-denting_relative-Daugavet}
Let $X$ be a Banach space and $x\in S_X$. If $x$ is a
Daugavet-point relative to the supporting slice $S:=S(x^*,\alpha)$, then
$\norm{x-z}=2$ for every $z\in S\cap \dent{B_X}$.
\end{prop}

\begin{proof}
Let $z\in S\cap \dent{B_X}$. Since $S$ is norm open, we can find
$\delta>0$ such that $z+\delta B_X\subset S$. In particular, any slice
$T$ of $B_X$ containing $z$ and of diameter $\eps\in(0,\delta]$
is a subslice of $S$, and by assumption, contains an
element $y\in T$ satisfying $\norm{x-y}\geq
2-\eps$. Then $\norm{x-z}\geq \norm{x-y}-\norm{y-z}\geq
2-2\eps$, and the conclusion follows by letting $\eps$ tend to $0$.
\end{proof}

Second, note that the previous result yields a characterization of
relative Daugavet-points in spaces with the RNP. This generalizes the
observations from \cite[Section~3]{VeeorgStudia}.

\begin{prop}\label{prop:RNP_relative-Daugavet_characterization}
Let $X$ be an RNP space and $x\in S_X$. Then $x$ is a
Daugavet-point relative to the slice $S:=S(x^*,\alpha)$ if and only
if $x$ is at distance $2$ from every denting point of $B_X$ which belongs
to $S$.
\end{prop}

\begin{proof}
By Proposition~\ref{prop:distance-to-denting_relative-Daugavet}, $x$
has to be at distance $2$ from every denting point of $B_X$ that
belongs to $S$ in order to be a Daugavet-point relative to $S$. So let
us assume that $x$ is at distance $2$ from every element of $S\cap
\dent{B_X}$, and pick any subslice $T$ of $S$. Since $X$ has the RNP,
$B_X=\cconv(\dent{B_X})$ and $T$ must contain a denting point $y$ of
$B_X$. By assumption, $\norm{x-y}=2$, and the conclusion follows. Note
that in this case an exact distance $2$ to subslices of $S$ is always
attained.
\end{proof}

Finally, let us show that a similar characterization also holds in
Lipschitz-free spaces. This provides a relative version of \cite[Theorem~2.1]{VeeorgStudia}, 
the latter result being a generalization of \cite[Theorem~3.2]{JungRueda}.

\begin{thm}\label{thm:lipfree_relative-Daugavet_characterization}
    Let $M$ be a metric space and $\mu\in S_{\lipfree{M}}$. Then $\mu$
    is a Daugavet-point relative to the slice $S:=S(f,\alpha)$ if and
    only if $\mu$ is at distance $2$ from every denting point of
    $B_{\lipfree{M}}$ which belongs to $S$.
\end{thm}

\begin{proof}
Let $\mu\in S_{\lipfree{M}}$, and assume that there exists a slice
$S:=S(f,\alpha)$ with $f(\mu)=1$ such that $\|\mu-\nu\|=2$ for every
$\nu \in S\cap \dent{B_{\lipfree{M}}}$. Fix $\varepsilon>0$ and a
subslice $T:=S(g,\beta)$ of $S$. If $g$ is local then by
{\cite[Theorem~2.6]{JungRueda}} there exists $m_{uv}\in T$ such that
$\|\mu-m_{uv}\|\ge 2-\varepsilon$.
If $g$ is not local, then by {\cite[Proposition~2.7]{VeeorgFunc}}, $T$
contains a denting point $m_{uv}$ of $B_{\lipfree{M}}$. By assumption
we have $\|\mu-m_{uv}\|=2$, therefore $\mu$ is a Daugavet-point
relative to $S$.
\end{proof}

With these tools at hand, we can now discuss the following two
examples, which are obtained by slight modifications of Veeorg's space
$\mathcal{M}$ \cite[Example~3.1]{VeeorgStudia}.
In the first example we remove one point from $\mathcal{M}$
and get a Lipschitz-free space with a relative Daugavet-point
that is not a Daugavet-point.
The second provides a $\Delta$-point that is not a relative
Daugavet-point and is constructed by removing even more points
from $\mathcal{M}$.

Note that the metric space $M$ in both examples share the property
with \cite[Example~3.1]{VeeorgStudia} that denting points of
the unit ball of $\lipfree{M}$ are exactly the molecules $m_{pq}$
with $p,q\notin\set{x,y}$ and $[p,q]=\set{p,q}$, where
$[p,q]$ is the metric line segment between $p$ and $q$.
This can be proved using \cite[Corollary~3.44]{Weaver2},
\cite[Theorem~4.1]{AG19}, and \cite[Theorem~2.4]{GLPPRZ}.
Arguing as in \cite[Theorem~2.1]{AALMPPV} we also have
that $\lipfree{M}$ is a separable dual space isomorphic to $\ell_1$
in both examples.

\begin{example}\label{expl:relative-Daugavet_not_Daugavet}
  Let $x:=(0,0)$, $y:=(1,0)$, $u:=(0,1/2)$,
  $v:=(1,1/2)$, $S_0:=\{x,y\}$ and $S_1:=\{u,v\}$.
  For every $n\in \mathbb{N}\setminus\{1\}$ let
  \begin{equation*}
    S_{n} := \Big\{
    \Big( \frac{k}{2^n},\frac{1}{2^{n}} \Big)
    \colon k\in\{0,1,\ldots,2^n\}
    \Big\}.
  \end{equation*}
  Consider
  $M := \bigcup_{n=0}^\infty S_n$
  with the metric
  \begin{equation*}
    d\big((a_1,b_1),(a_2,b_2)\big)
    :=
    \begin{cases}
      |a_1-a_2|, & \text{if }b_1=b_2,\\
      \min\{a_1+a_2,2-(a_1+a_2)\}+|b_1-b_2|, &\text{if } b_1\neq b_2.
    \end{cases}
  \end{equation*}
  See
  Figure~\ref{figure:expl_relative-Daugavet_not_Daugavet}
  for a picture of $M$.
  From \cite[Theorem~4.1]{AG19} we get that $m_{uv}$ is a denting
  point, and by using \cite[Lemma~1.2]{VeeorgStudia} we get
  $\|m_{xy}-m_{uv}\|<2$, therefore $m_{xy}$ is not a Daugavet-point.

  Define $f\colon\set{x,y,u,v}\rightarrow \R$ by $f(x)=0$, $f(y)=1$,
  and $f(u)=f(v)=1/2$. It is straightforward to check that the
  Lipschitz constant of $f$ is $1$. Now extend $f$ by the McShane--Whitney
  Theorem to $M$ and let $S := S(f,1)$. Then $m_{uv}\notin S$.

  Arguing as in \cite[Example~3.1]{VeeorgStudia} we get that $m_{xy}$
  is at distance $2$ from every denting point of $B_{\lipfree{M}}$
  except $m_{uv}$.
  It follows from
  Theorem~\ref{thm:lipfree_relative-Daugavet_characterization}
  that $m_{xy}$ is a Daugavet-point relative to $S$.
\end{example}

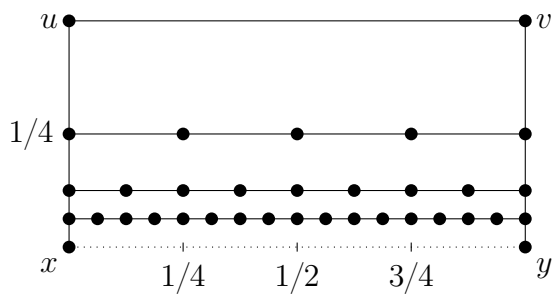
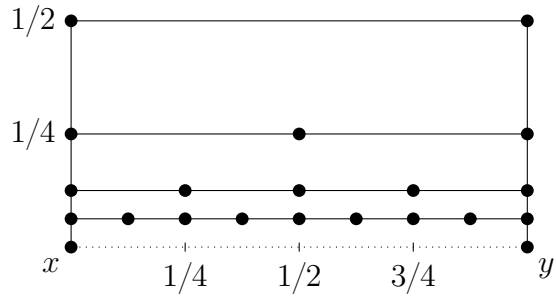
\begin{figure}[ht]
  \centering
  \begin{subfigure}{0.45\textwidth}
    \begin{tikzpicture}[scale=6.0]
      \draw[dotted] (-0.01,0) -- (1.0,0);
      \draw (0,-0.01) -- (0,0.5);
      \draw (1,0) -- (1,1/2);

      \foreach \y in {1/4} {
        \draw (0.01,\y) -- (-0.01,\y) node[left] {$\y$};
      }

      \foreach \x in {1/4, 1/2, 3/4} {
        \draw (\x,0.01) -- (\x,-0.01) node[below] {$\x$};
      }

      \fill (0,0) circle[radius=0.4pt] node[below left] {$x$};
      \fill (1,0) circle[radius=0.4pt] node[below right] {$y$};

      \fill (0,0.5) circle[radius=0.4pt] node[left] {$u$};
      \fill (1,0.5) circle[radius=0.4pt] node[right] {$v$};

      \draw (0,{2^(-1)}) -- (1,{2^(-1)});
      \foreach \sn in {2,3,4} {
        \draw (0,{2^(-\sn)}) -- (1,{2^(-\sn)});

        \pgfmathtruncatemacro\aa{2^\sn}
        \foreach \k in {0,...,\aa} {
          \fill ({\k*2^(-\sn)},{2^(-\sn)}) circle[radius=0.4pt];
        }
      }
    \end{tikzpicture}
    \caption{The sets $S_0,\ldots,S_4$ from
      Example~\ref{expl:relative-Daugavet_not_Daugavet}}
    \label{figure:expl_relative-Daugavet_not_Daugavet}
  \end{subfigure}
  \hfill
  \begin{subfigure}{0.45\textwidth}
    \begin{tikzpicture}[scale=6.0]
      \draw[dotted] (-0.01,0) -- (1.0,0);
      \draw (0,-0.01) -- (0,0.5);
      \draw (1,0) -- (1,1/2);

      \foreach \y in {1/4, 1/2} {
        \draw (0.01,\y) -- (-0.01,\y) node[left] {$\y$};
      }

      \foreach \x in {1/4, 1/2, 3/4} {
        \draw (\x,0.01) -- (\x,-0.01) node[below] {$\x$};
      }

      \fill (0,0) circle[radius=0.4pt] node[below left] {$x$};
      \fill (1,0) circle[radius=0.4pt] node[below right] {$y$};

      \foreach \sn in {1,2,3,4} {
        \draw (0,{2^(-\sn)}) -- (1,{2^(-\sn)});

        \pgfmathtruncatemacro\aa{2^(\sn-1)}
        \foreach \k in {0,...,\aa} {
          \fill ({\k*2^(-\sn+1)},{2^(-\sn)}) circle[radius=0.4pt];
        }
      }
    \end{tikzpicture}
    \caption{The sets $S_0,\ldots,S_4$ from
      Example~\ref{expl:Delta_not_relative-Daugavet}}
    \label{figure:expl_Delta_not_relative-Daugavet}
  \end{subfigure}
  \caption{Comparison of
    Examples~\ref{expl:relative-Daugavet_not_Daugavet}
    and
    \ref{expl:Delta_not_relative-Daugavet}}
  \label{figure:separation}
\end{figure}

\begin{example}\label{expl:Delta_not_relative-Daugavet}
  Let $x:=(0,0)$, $y:=(1,0)$, and $S_0:=\{x,y\}$.
  For every $n\in \mathbb{N}$ let
  \begin{equation*}
    S_{n}
    :=
    \Big\{
    \Big( \frac{k}{2^{n-1}},\frac{1}{2^{n}} \Big)
    \colon k\in\{0,1,\ldots,2^{n-1}
    \}
    \Big\}.
  \end{equation*}
  Consider
  $M :=\bigcup_{n=0}^\infty S_n$
  with the metric
  \begin{equation*}
    d \big((a_1,b_1),(a_2,b_2) \big)
    :=
    \begin{cases}
      |a_1-a_2|, & \text{if } b_1=b_2,\\
      \min\{a_1 + a_2, 2 - (a_1+a_2)\} + |b_1 - b_2|,
      & \text{if } b_1 \neq b_2.
    \end{cases}
  \end{equation*}
  See
  Figure~\ref{figure:expl_Delta_not_relative-Daugavet}
  for a picture of $M$.
  It is clear that the points $x$ and $y$ are \emph{discretely
    connectable} in $M$ in the sense of \cite[Definition~6.1]{AALMPPV},
  so by \cite[Theorem~6.7]{AALMPPV} $m_{xy}$ is a $\Delta$-point.

  Now choose any slice $S(f,\alpha)$ with $f(m_{xy})=1$.
  Then there exist $n\in \mathbb{N}$ and $u\neq v\in S_n$ such that
  $m_{uv}\in S(f,\alpha)$ and $[u,v]=\{u,v\}$.
  Then $m_{uv}$ is a denting point and $\|m_{xy}-m_{uv}\|<2$.
  Therefore $m_{xy}$ is not a relative Daugavet-point.
\end{example}

\begin{rem}
  Another explicit example of a $\Delta$-point that is not a
  relative Daugavet-point can be found in the equivalent renorming $Y$
  of $\ell_2$ from \cite[Section~3]{AALMPPV}. Indeed, it is proved
  there that the first unit basis vector $e_1$ of $\ell_2$ is a
  $\Delta$-point in $Y$, but that there is a sequence of strongly
  exposed points in $B_Y$ that converges in norm to $e_1$. It clearly
  follows that $e_1$ is not a relative Daugavet-point.
\end{rem}

Using ideas from \cite{HLPV23}, we slightly modify the renorming from
\cite[Section~4]{AALMPPV} that was inspired by the one from
\cite[Theorem~2.4]{DKR+}, in order to renorm every infinite
dimensional Banach space with a weakly null Schauder basis in such a
way that it has a relative Daugavet-point.

\begin{thm}\label{thm:relative-Daugavet_renorming}
    Let $X$  be an infinite dimensional Banach space with a weakly
    null Schauder basis $(e_n)$. Then $X$ can be renormed with a
    relative Daugavet-point.
\end{thm}

\begin{proof}
   Let us assume, as we may, that $(e_n)$ is normalized, and let
   $(e_n^*)$ be the associated sequence of biorthogonal
   functionals. Let $Y:=\cspn(e_n)_{n\geq 2}$ and let $\nnorm{\cdot}$
   be the equivalent norm on $X$ whose unit ball
   is $$B_{(X,\nnorm{\cdot})}:=\cconv\left(B_Y\cup\{\pm e_1\}\cup
     \{\pm(e_1+2e_n),\ n\geq 2\}\right).$$ Note
   that $$\cconv\left(B_Y\cup\{\pm e_1\}\right)=B_{Y\oplus_1 \R
     e_1},$$ so $\nnorm{\cdot}$ is indeed an equivalent norm on
   $X$. Also, we clearly
   have $$\nnorm{e_1}=\nnorm{e_n}=\nnorm{e_1+2e_n}=1$$ for every
   $n\geq 2$.
    
    By construction, every slice of $B_{(X,\nnorm{\cdot})}$ intersects
    either $\{\pm e_1\}$, $\{\pm(e_1+2e_n),\ n\geq 2\}$, or $S_Y$. So
    consider $S:=S(e_1^*,\frac{1}{2})$. Then $S$ does not intersect
    $Y=\ker e_1^*$, and  neither $-e_1$ nor $-(e_1+2e_n)$ belong to
    $S$. Let $T$ be a subslice of $S$. By the previous observation,
    $T$  contains either $e_1$ or $e_1+2e_n$ for some $n\geq 2$. By
    assumption, we have that $e_n\to 0$ weakly, so $e_1+2e_n\to e_1$
    weakly, and as slices of $B$ are relatively weakly open, we get in
    either case that  $T$ contains $e_1+2e_n$ for some $n\geq 2$. As
    $\nnorm{e_1-(e_1+2e_n)}=2\nnorm{e_n}=2$, we conclude that $e_1$ is
    a Daugavet-point relative to $S$.
\end{proof}

\begin{rem}
    In general, it seems that this renorming does not turn the point
    $e_1$ into a Daugavet-point. Indeed, if e.g. $X:=\ell_2$, then it
    is straightforward to check that $x:=-\frac{1}{4}\sum_{i=2}^{17}e_i\in S_{(\ell_2,\nnorm{.})}$ is strongly exposed by $-\frac{1}{4}\sum_{i=2}^{17}e_i^*\in S_{(\ell_2,\nnorm{.})^*}$, and that $x$ is at distance strictly less than 2 from $e_1$.
    Yet unlike \cite{AALMPPV},  the point $e_1$ has been carefully sent
    away from $Y$, so that these strongly exposed point cannot get too
    close to it anymore, and can actually be separated from $e_1$ by
    the means of a supporting slice. Dealing with this rather delicate
    obstruction in order to get a Daugavet renorming requires the full
    power of the construction from \cite{HLPV23}, which heavily relies
    on the unconditionality of the basis.
\end{rem}

Anticipating the transfer results for relative Daugavet-points from
Section~\ref{subsec:relative_Daugavet-property}, let us finally note
that, similar to \cite{HLPV23}, we can then use the previous renorming
result to get the following corollary.

\begin{cor}\label{cor:complemented_relative-Daugavet_renorming}
Let $X$ be an infinite dimensional Banach space. If $X$ contains a
weakly null complemented basic sequence $(e_n)$, then $X$ can be
renormed with a relative Daugavet-point.
\end{cor}

\begin{proof}
    Let $E:=\cspn(e_n)_{n\geq 1}$. By
    Theorem~\ref{thm:relative-Daugavet_renorming}, there exists an equivalent
    norm $\nnorm{\cdot}$ on $E$ for which $E$ contains a
    relative Daugavet-point $x$. As $E$ is complemented in $X$, there exists a
    Banach space $Z$ such that $X$ is isomorphic to
    $Y:=(E,\nnorm{\cdot})\oplus_1 Z$. By 
    Corollary~\ref{cor:polyhedral_stability_relative_Daugavet-points},
    $(x,0)$ is a relative Daugavet-point in $Y$.
\end{proof}

The above still leaves open the following question.

\begin{quest}
    Can every infinite dimensional Banach space be renormed with a
    relative Daugavet-point?
\end{quest}


\subsection{Convex
  representations of relative
  Daugavet-points}\label{subsec:convex_representations}

As we discussed in Section~\ref{sec:preliminaries}, it follows from
the geometric characterizations from
\cite{KSSW} that every Banach space with the Daugavet property
satisfies the LD2P, hence that no RNP space admits an equivalent
norm with the Daugavet property.  It is thus very natural to want to
understand whether the existence of Daugavet-points in a given Banach
space has a similar influence on its isomorphic geometry. This
question was first answered negatively in \cite{VeeorgStudia} where an
example of a Lipschitz-free space with the RNP and with a
Daugavet-point was provided. Note that it was recently proved in
\cite[Theorem~2.1]{AALMPPV} that this space is actually isomorphic to
$\ell_1$, and that it is (isometrically) a separable dual space. Also,
it was proved in \cite{HLPV23} that the space $\ell_2$ can be renormed
with a point that is simultaneously a super Daugavet-point and an
extreme point of the unit ball (hence also a ccw $\Delta$-point in the
sense of \cite{MPRZ}) which is the strongest diametral notion that a
point of the unit sphere of a strongly regular Banach space can satisfy.
Let us emphasize once more that these examples are
particularly striking in view of the observation from \cite{JungRueda}
that Daugavet-points have to be at distance 2 from all the denting
points of the unit ball in any given Banach space, and that super
Daugavet-points
have to be at distance 2 from all the points of continuity
(\cite[Lemma~3.7]{MPRZ}).
In this section we aim at shedding more light on the interplay
between (relative) Daugavet-points and denting points.

As relative Daugavet-points behave locally very much like
Daugavet-points, it turns out that all the results from this section,
which are closely related to the distance 2 to denting points, can be
stated both for Daugavet-points and their relative counterparts. We
chose to go with the second option, but let us remind the reader that
thanks to Lemma~\ref{lem:Daugavet_implies_relative_Daugavet}, the
specific statements for Daugavet-points will naturally be included.

We start by the following elementary observation about the
relative Daugavet-points which can be represented, in a given Banach
space, as \emph{convex combinations} or as \emph{convex series} of
elements of the unit ball.

\begin{prop}\label{prop:Daugavet_convex_series}
  Let $X$ be Banach space, $x\in S_X$, and $(z_n)_{n\in I}$ be
  a family of elements of $B_X$ for some non-empty subset $I$ of
  $\mathbb{N}$. Then let us assume that there exists
  $(\lambda_n)_{n\in I} \subset (0,1]$ with $\sum_{n \in I}
  \lambda_n = 1$, and such that $x = \sum_{n \in I} \lambda_n z_n$. If
  $x$ is a Daugavet-point relative to some supporting slice
  $S:=S(x^*,\alpha)$, then each $z_n$ is a Daugavet-point relative to
  $S$.
\end{prop}

\begin{proof}
   First note that since $x^*(x)=1$, we must have $x^*(z_k)=1$ for
   every $k\in I$, so $x^*$ is a supporting functional for each $z_k$.
   Fix $k \in I$ and $\eps>0$, and pick any subslice $T$ of $S$.  Since
   $x$ is a Daugavet-point relative to $S$, we can find $y\in T$ such that
   $\norm{x-y}>2-\lambda_k\eps$. We have
  \begin{align*}
    \lambda_k\|z_k - y\|
    &\ge \left\|\sum_{n \in I} \lambda_n z_n - \sum_{n \in I} \lambda_n y\right\|
    - \left\|\sum_{n \in I\backslash\{k\}} \lambda_n z_n - \sum_{n \in
      I\backslash\{k\}} \lambda_n y\right\|\\
    & \ge \|x - y\| - \sum_{n \in I\backslash\{k\}} \lambda_n \|z_n - y\|\\
    & > 2 - \eps\lambda_k - 2(1 - \lambda_k) \\
   & = (2 - \eps)\lambda_k,
  \end{align*}
  so $\|z_k - y\| > 2 - \eps$, and it follows that $z_k$ is a
  Daugavet-point relative to $S$.
\end{proof}

\begin{obs}\label{obs:convex_series_of_Daugavet_points}
In general, the converse of
Proposition~\ref{prop:Daugavet_convex_series} does not hold true. For
example, it follows from \cite[Theorem~3.4]{AHLP} that the
Daugavet-points in the
space $c$ of convergent sequences are exactly the elements with limits
$\pm 1$. But obviously every finitely supported element in $S_c$ can
be written as a convex combination of two such points.
\end{obs}

Since no relative Daugavet-point is denting, we immediately obtain the
following corollary.

\begin{cor}\label{cor:convex_combinations_denting}
    No relative Daugavet-point can be represented as a convex
    combination or as a convex series of denting points.
\end{cor}

Note that this allows to give a very short proof of the non-existence
of relative Daugavet-points in finite dimensional spaces.

\begin{thm}\label{thm:finite_dim_Daugavet}
    Let $X$ be a finite dimensional Banach space. Then $X$ contains no
    relative Daugavet-point.
\end{thm}

\begin{proof}
    By Minkowski's theorem we can represent any element of $B_X$ as a
    convex combination of ($n=\dim(X)+1$ by Carath\'{e}odory's
    theorem) extreme points of $B_X$. Since $X$ has finite dimension,
    every extreme point is denting, so by Corollary
    \ref{cor:convex_combinations_denting}, $X$ contains no
    relative Daugavet-point.
\end{proof}

\begin{rem}
  In general, $\Delta$- and $\mathfrak{D}$-points behave very
  differently in comparison with relative Daugavet-points.
  For example \cite[Example~4.6]{VeeorgStudia} provides
  a $\Delta$-point which can be represented both as an
  average of two $\Delta$-points and as an average of
  two denting points,
  and \cite[Theorem~3.1]{AALMPPV} provides a $\Delta$-point
  (in fact a \emph{ccw $\Delta$-point} in the terminology of
  \cite{MPRZ}) which is the limit of a sequence of strongly exposed
  points.
  The fact that finite dimensional spaces fail to contain
  $\mathfrak{D}$-points was implicitly proved in \cite{ALMP} (see
  \cite[Corollary~5.3]{AALMPPV}), but the result is more involved and
  its proof has, as expected, a very different flavor.
\end{rem}

The two key points in the previous proof are Minkowski's representation
theorem on one side, and on the other side the well known fact that
every extreme point in the unit ball of a finite dimensional space is
automatically denting. Note that the latter immediately follows from
the result by Lin, Lin and Troyanski \cite{LLT} which states that a
point $x$ in the unit sphere of a given Banach space $X$ is denting if
and only if it is an extreme point of $B_X$ and a point of
(weak-to-norm) continuity.

It is in general not true in infinite dimensional spaces, and even
among those with the RNP, that any given element of the unit
ball can be represented as a convex series of extreme points. However,
the notion of barycenters is known, since Choquet's results about
integral extremal representations on compact convex sets, to provide a
fruitful generalization of this notion in infinite dimension (we refer
e.g. to \cite{Phelps01} for a detailed introduction to the subject).

It turns out that Corollary \ref{cor:convex_combinations_denting} admits
a natural generalization, in separable Banach space, for barycenters.

\begin{lem}\label{lem:barycenter_denting}

Let $X$ be a separable Banach space and $x\in S_X$. If $x$ can be
represented as the barycenter of a Borel probability measure $\mu$ on
$B_X$ which is supported on the set $\dent{B_X}$, then $x$ is not a
relative Daugavet-point.
    
\end{lem}

\begin{proof}

In this setting, it is well known that we can write $$x=\int_{B_X} y\
d\mu$$ as a Bochner integral (see
e.g. \cite[Lemma~6.2.2]{Bourgin}). In particular, note that any
supporting slice $S$ for the point $x$ contains denting points of
$B_X$, as the intersection of $S$ with the support of $\mu$ has
non-zero measure (by convexity of $B_X\backslash S$), and as $\mu$ is
supported on $\dent{B_X}$. We will show that $x$ has to be at distance
strictly less than $2$ from one of these. Indeed, assume for a
contradiction that this does not hold. Then for every $z\in S\cap
\dent{B_X}$, we have $$2=\norm{x-z}=\norm{\int_{B_X}y-z\ d\mu}\leq
\int_{B_X}\norm{y-z}d\mu\leq 2.$$ Hence $$\int_{B_X}\norm{y-z}d\mu=
2,$$ and as a consequence, the set $$A_z:=\left\{y\in B_X:\
  \norm{y-z}<2\right\}$$ has measure 0.  Now pick a dense sequence
$(z_n)_{n\geq 1}$ in $S\cap\supp{\mu}$. Then, by assumption, every
$z_n$ is denting, and by countable sub-additivity, the set
$A:=\bigcup_{n\geq 1}A_{z_n}$ also has measure $0$. It follows that
$B:=(\supp{\mu})\backslash A$ has non-zero measure, and in particular is
non-empty. But if $y\in B$, then $\norm{y-z_n}=2$ for every $n\geq 1$,
hence $\norm{y-z}=2$ for every $z\in \supp{\mu}$ by density. Since $y$
itself belongs to $\supp{\mu}$, we have a contradiction.
\end{proof}

In separable RNP spaces, Edgar's ``non-compact Choquet theorem"
\cite{Edgar75} allows to represent any given point of the unit ball
as the barycenter of a Borel probabilty measure supported on the set
of extreme points. So under the assumption that all the extreme points
of a given separable RNP space are denting, we immediately get the
following extension of Theorem~\ref{thm:finite_dim_Daugavet}.

\begin{thm}\label{thm:separable_RNP_denting}
Let  $X$ be a separable Banach space. If $X$ has the RNP and every
extreme point of $B_X$ is denting, then $X$ does not contain
relative Daugavet-points.
\end{thm}

\begin{proof}
    Let $x\in S_X$. By Edgar's theorem, $x$ can be represented as the
    barycenter of a Borel probability measure $\mu$ on $B_X$ supported
    on $\ext{B_X}$. Since by assumption $\ext{B_X}=\dent{B_X}$, we get
    by Lemma~\ref{lem:barycenter_denting} that $x$ is not a
    relative Daugavet-point.
\end{proof}

We will provide a non-separable and formally stronger version of this
result below (Theorem~\ref{thm:KMP_denting}), but let us first give a
few comments. In general, extreme points do not have to be denting. In
particular, note that the Daugavet-molecule in the example from
\cite{VeeorgStudia} as well as the Daugavet-points produced in \cite{HLPV23}
are all extreme points, and this perfectly
illustrates the failure of a general purely isomorphic version of the
previous statement. In fact we will now show that these examples are not
isolated, because every RNP space with a relative Daugavet-point must
contain an element which is simultaneously a relative Daugavet-point and an
extreme point of the unit ball.

The proof of this fact will actually rely on considerations about a
specific face of the unit ball and involve the KMP. But in order to
get a general statement in this context, we will need to additionally
consider the following property.

\begin{defn}\label{defn:dtd_Daugavet}

We say that a Banach space $X$ satisfies the \emph{distance-to-denting
  property for relative Daugavet-points} if for every element $x\in S_X$ and
for every supporting slice $S:=S(x^*,\alpha)$, we have that $x$ is at
distance 2 from all the denting points of $B_X$ in $S$ if and only if
$x$ is a Daugavet-point relative to $S$.
    
\end{defn}

With this terminology,
Proposition~\ref{prop:distance-to-denting_relative-Daugavet} states
that every RNP space satisfies the distance-to-denting property for
relative Daugavet-points, and
Theorem~\ref{thm:lipfree_relative-Daugavet_characterization} that so
does every Lipschitz-free space.  Note that a closer look at
\cite[Theorem~3.1]{AHLP} shows that the same goes in any
$L_1(\mu)$-space. However, there  exist examples of Banach spaces
where the distance-to-denting characterization fails ($c$ or $\ell_\infty$ are
sort of trivial examples of this kind since
$\dent{B_c}=\emptyset$ and $\dent{B_{\ell_\infty}}=\emptyset$ ).

We can now prove the following result.

\begin{prop}\label{prop:extreme_daugavet_KMP}
Let $X$ be a Banach space with the KMP that satisfies the
distance-to-denting characterization for relative Daugavet-points (e.g. $X$ with the RNP). 
If $X$ contains a Daugavet-point relative to a supporting slice $S$, then $X$ contains a
Daugavet-point relative to $S$ that is also an extreme point of $B_X$.
\end{prop}

\begin{proof}
Assume that $x \in S_X$ is a Daugavet-point relative to
$S:=S(y^*,\alpha)$, $y^* \in D(x)$ and $\alpha > 0$, and define
\begin{equation*}
  F_x:=\bigcap_{x^*\in D(x)}\left\{y\in B_X: x^*(y)=x^*(x)=1\right\}.
\end{equation*}
Observe that $F_x$ is a non-empty face of $B_X$.

Let $y \in F_x$. Then $y^* \in D(y)$.
We will first prove that $y$ is a Daugavet-point relative to $S$.
Since $x$ is a Daugavet-point
relative to $S$, we have that $\|x-z\|=2$ for every $z\in S\cap
\dent{B_X}$. Now by Hahn--Banach, we can find for any such $z$ a
functional $z^*\in S_{X^*}$ such that $z^*(x-z)=2$. Then
$z^*(x)=-z^*(z)=1$, and in
particular $z^*\in D(x)$. Since $y\in F_x$, we also have $z^*(y)=1$,
and this gives us \[\|y-z\|\ge z^*(y-z)=2.\]
Since $X$ satisfies the distance-to-denting characterization for
relative Daugavet-points, it follows that $y$ is a Daugavet-point
relative to $S$ as stated.

Now, since $F_x$ is a non-empty, closed, bounded, and convex subset
of $X$, and since $X$ has the KMP, then $F_x$ admits an
extreme point $y$.  As $F_x$ is a face, $y$ is also an extreme
point of $B_X$. Thus we have found a Daugavet-point relative to $S$
that is also an
extreme point of $B_X$.
\end{proof}

In particular, note that this immediately yields a non-separable
version of Theorem~\ref{thm:separable_RNP_denting}. Moreover, observe
that, using similar ideas, we can provide the following (formally)
stronger statement.

\begin{thm}\label{thm:KMP_denting}
Let  $X$ be a Banach space. If $X$ has the KMP, and if every extreme
point of $B_X$ is denting, then $X$ fails to contain relative
Daugavet-points.
\end{thm}

\begin{proof}
If $X$ has the KMP, and if every extreme
point of $B_X$ is denting, then $B_X$ is the closed convex hull of the
set of all denting points of $B_X$ and thus $X$ satisfies the
distance-to-denting characterization for relative Daugavet-points. The
conclusion then immediately follows from
Proposition~\ref{prop:extreme_daugavet_KMP}.
\end{proof}

\begin{rem}
 {\ }

    \begin{enumerate}
    \item By the result from Lin, Lin and Troyanski that was previously
    mentioned, it is enough (in fact equivalent) to assume that every
    extreme point of $B_X$ is a point of continuity in the above
    statement. In particular, any \emph{asymptotic uniformly convex}
    (\emph{AUC}) space with the KMP fails to contain Daugavet-points
    since it is well known, see e.g. the proof of
    \cite[Proposition~2.6]{JLPS}, that every point on the unit sphere
    of an AUC space is a point of continuity. Note that the dual of
    the James tree spaces $JT^*$ is known to be AUC \cite{Girardi2001}
    but that it fails the KMP because $JT$ is not Asplund (it is a
    separable space with a non-separable dual).

    \item In \cite[Lemma~3.7]{MPRZ}, it was observed that super Daugavet-points have a
      similar relationship with points of continuity. More precisely, every
      super Daugavet-point has to be at distance 2 from all the points
      of continuity of the ball, and the converse does hold true in
      spaces with the \emph{convex point of continuity property}
      (\emph{CPCP}). It is currently unknown whether this
      characterization extends to other classes of Banach spaces.
    \end{enumerate}
    
\end{rem}


\subsection{Stability results and relative Daugavet
  property}\label{subsec:relative_Daugavet-property}

In this section we study the stability of relative Daugavet-points
under the operation of taking absolute sums of Banach spaces. We
provide new examples of relative Daugavet-points which are not
Daugavet, and of $\Delta$-points which are not relative Daugavet. Also
we introduce a new property of Banach spaces that we call relative
Daugavet property, and that we define by asking all the elements of
the unit sphere to be relative Daugavet. Using the stability results,
we show that this property lies strictly between the Daugavet property
and the DLD2P, and we provide an example of a space with the relative
Daugavet property that contains no Daugavet-point, and an example of a space with
the DLD2P that contains no relative Daugavet-point.

Recall that a norm $N$ on $\mathbb{R}^2$ is \emph{absolute} if
$N(a, b) = N(|a|, |b|)$ for all $(a,b) \in \mathbb{R}^2$, and
\emph{normalized} if $N(1, 0) = N(0,1) = 1.$  The $\ell_p$-norms are
examples of absolute normalized norms on $\mathbb{R}^2$ for 
$ 1 \le p \le \infty.$ For any given absolute normalized norm $N$ on
$\R^2$, we have $N(a,b)\leq N(c,d)$ whenever $\abs{a}\leq \abs{c}$ and
$ \abs{b}\leq \abs{d}$. In particular,  $$\norm{\cdot}_\infty\leq
N\leq \norm{\cdot}_1.$$ If $X$ and $Y$ are Banach
spaces and if $N$ is an absolute norm on $\mathbb{R}^2,$ then we denote
by $X \oplus_N Y$ the direct sum $X \oplus Y$ with the norm $\|(x, y)\|_N =
N(\|x\|, \|y\|)$ and we call this Banach space the \emph{absolute sum}
of $X$ and $Y$.  In the case of an $\ell_p$-norm, we simply write $X \oplus_p Y$
the \emph{$\ell_p$-sum} of $X$ and $Y$. Note that if $N$ is an
absolute normalized norm on $\mathbb{R}^2$, then the dual norm of $N$
is the absolute normalized  norm $N^*$ on $\R^2$ given by the
formula $$N^*(c,d)=\max\{a\abs{c}+b\abs{d}:\ a,b\geq 0,\ N(a,b)=1\}.$$
It is also well known that the dual of the space  $X\oplus_N Y$ can be
identified with the space $X^*\oplus_{N^*} Y^*$.

While the DLD2P is stable by any absolute norm \cite{zbMATH02168839},
the Daugavet property is only stable by the $\ell_1$- and
$\ell_\infty$-norms \cite{BKSW}.
The study of which absolute normalized norms allow Daugavet-points
was started in \cite{AHLP} and finally clarified in \cite{HPV}.
For this purpose, two disjoint families of absolute norms were 
introduced: One consists in those absolute norms with the so called
property $(\alpha)$
\cite[Definition~4.4]{AHLP}, and the other in A-octahedral norms
\cite[Definition~2.1]{HPV}. Recall that every absolute normalized  
norm on $\R^2$ belongs to one of the two previous classes
\cite[Proposition~2.5]{HPV}, and that
$N$-sums of Banach spaces do not contain Daugavet-points if 
$N$ has property $(\alpha)$ \cite[Proposition~4.6]{AHLP}. 
Positive stability results for Daugavet-points were obtained in
\cite[Theorem~2.2]{HPV} for A-octahedral norms.

We will see later on that positive stability results for relative
Daugavet-points can be obtained in a much larger class of absolute
normalized norms, namely those for which the unit ball of the
underlying two-dimensional space presents a kind of local
polyhedrality. Recall that a normed space $X$ is said to be
\emph{polyhedral} if its unit ball $B_X$ is the intersection of a
finite number of half-planes. Examples of two-dimensional polyhedral
spaces with absolute normalized norms are $\ell_1^2$ and
$\ell_\infty^2$, and more generally any two-dimensional space with an
A-octahedral norm.

For the study of the stability of relative Daugavet-points, we will
thus need the following key definition.

\begin{defn}
  Let $X = (\R^2,\|\cdot\|)$ be a two-dimensional Banach space
  and let $x \in S_X$.
  We say that $x$ is a \emph{vertex-point} (\emph{v-point})
  if there exist $y,z \in S_Z$ such that
  $\|x + y\| = \|x + z\| = 2$ and $\|y + z\| < 2$.
  See Figure~\ref{fig:v-points}.
\end{defn}

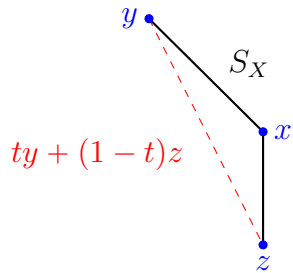
\begin{figure}[h]
  \centering
  \begin{tikzpicture}[scale=1.5]
 \draw[black,thick] (1,0) -- (0,1); 
 \draw[black,thick] (1,0) -- (1,-1);
 \draw[red,dashed] (0,1) -- (1,-1);
 \draw[black] (0.6,0.6) node[anchor=west]{$S_X$};
 \filldraw[blue] (1,0) circle (1pt) node[anchor=west]{$x$};
 \filldraw[blue] (0,1) circle (1pt) node[anchor=east]{$y$};
 \filldraw[blue] (1,-1) circle (1pt) node[anchor=north]{$z$};
 \filldraw[red] (0.4,-0.2) node[anchor=east]{$ty+(1-t)z$};
  \end{tikzpicture}
  \caption{Typical picture for a v-point}
  \label{fig:v-points}
\end{figure}

\begin{rem}\label{rem:vertex-points}
  Clearly, every $v$-point is an extreme point of the unit ball,
  and strictly convex norms admit no v-points.
  In particular, $B_{\ell_p^2}$ admits no v-point for $1<p<\infty$.
  On the other hand, if $X$ is polyhedral, then every $x \in S_X$ is a
  convex combination of two (not necessarily distinct) v-points.
\end{rem}

We start by proving a few technical lemmata concerning the behavior of
v-points.

\begin{lem}\label{lem:char_vpt_sep_ext}
  Let $X$ be a two-dimensional Banach space
  and let $x \in \ext B_X$.
  Then $x$ is a v-point if and only if
  there exists $\varepsilon > 0$ such that
  $B(x,\varepsilon) \cap \ext B_X = \set{x}$.

  In particular, $X$ is polyhedral if and only if 
  every $x \in \ext B_X$ is a v-point.
\end{lem}

\begin{proof}
  If $x \in S_X$ is a v-point, then there exist
  $y_i \in S_X$, $i=1,2$, such that
  $\|x + y_i\| = 2$ and $\|y_1 - y_2\| < 2$.
  Let $\varepsilon :=1/2 \min_i \|x - y_i\|$.
  Then every $y \in B(x,\varepsilon) \cap S_X$
  is on the line between $x$ and one of the $y_i$'s.
  Hence $B(x,\varepsilon) \cap \ext B_X = \set{x}$.

  Conversely, assume that there exists $\varepsilon > 0$
  such that $B(x,\varepsilon) \cap \ext{B_X} = \set{x}$.
  Since $x$ is extreme, it is denting, so there exists
  a slice $S$ of $B_X$ with $x \in S$ such that $S \subset B(x,\varepsilon)$.
  Now take any $y \in  S \cap S_X$ distinct from $x$.
  Since $X$ is two-dimensional we can write
  $y = \lambda x_1 + (1-\lambda) x_2$ where $\lambda \in [0,1]$
  and $x_i \in \ext B_X$.
  Using the assumption, we may assume that $x_1 = x$.
  Since $y \in S_X$, this means that the whole line segment
  between $x$ and $x_2$ is on the sphere, and in particular, $\|x + y\| = 2.$
  As $S \cap S_X$ is not a line segment, there exist $y_1, y_2 \in  S \cap S_X$
  with $\|y_1 + y_2\| < 2,$ so $x$ is a
  v-point.

  Finally, since $X$ is two-dimensional, $\ext B_X$
  is closed and hence compact.
  If every extreme point is a v-point, then
  we get an open covering of $\ext B_X$,
  so $\ext B_X$ is finite and this means
  that $X$ is polyhedral.
\end{proof}

\begin{lem}\label{lem:relative-Daugavet_direct_sum_1and2}
  Let $X$ be a two-dimensional Banach space.
  If $x \in S_X$ is a convex combination
  of v-points $x_1$ and $x_2$, then there exist
  $\alpha > 0$ and $x^* \in D(x)$ such that for every subslice $S$ of
  $S(x^*,\alpha)$ we have either $x_1 \in S$ or $x_2 \in S$.
\end{lem}

\begin{proof}
  Write $x = \lambda x_1 + (1-\lambda)x_2$.

  \textbf{Case 1.}
  $\lambda \in \{0,1\}$, i.e. $x$ is a v-point.
  Then there exist $y_i \in S_X$, $i=1,2$, and $\alpha > 0$ such that
  $\|x + y_i\| = 2$ and $\|y_1 + y_2\| < 2 - 2\alpha$.
  Find $x_i^* \in S_{X^*}$ such that $x_i^*(x + y_i) = 2$
  and define $x^* = (x_1^* + x_2^*)/2$.
  If $i \neq j$, then
  \begin{equation*}
    x_i^*(y_j) = x_i^*(y_i + y_j) - x_i^*(y_i)
    < 2 - 2\alpha - 1 = 1 - 2\alpha
  \end{equation*}
  so that
  \begin{equation*}
    x^*(y_i) = \frac{1}{2}(x_i^*(y_i) + x_j^*(y_i))
    < \frac{1}{2}(2 - 2\alpha) = 1 - \alpha.
  \end{equation*}
  If $S$ is a subslice of $S(x^*,\alpha)$ and
  $z \in S \cap S_X$, then since $X$ is two-dimensional, $z$ is a convex
  combination of either $x$ and $y_1$ or of $x$ and $y_2$.
  Since $y_i \notin S$ we must have $x \in S_X$.

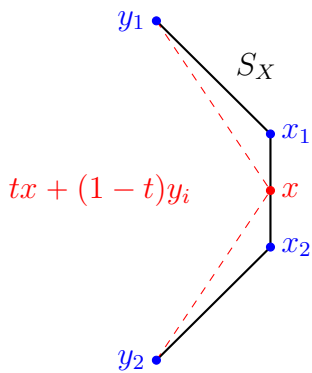
\begin{figure}[h]
  \centering
  \begin{tikzpicture}[scale=1.5]
 \draw[black,thick] (1,0) -- (0,1); 
 \draw[black,thick] (1,-1) -- (0,-2); 
 \draw[black,thick] (1,0) -- (1,-1);
 \draw[red,dashed] (0,1) -- (1,-0.5);
 \draw[red,dashed] (0,-2) -- (1,-0.5);
 \draw[black] (0.6,0.6) node[anchor=west]{$S_{X}$};
 \filldraw[red] (1,-0.5) circle (1pt) node[anchor=west]{$x$};
 \filldraw[blue] (1,0) circle (1pt) node[anchor=west]{$x_1$};
 \filldraw[blue] (1,-1) circle (1pt) node[anchor=west]{$x_2$};
 \filldraw[blue] (0,1) circle (1pt) node[anchor=east]{$y_1$};
 \filldraw[blue] (0,-2) circle (1pt) node[anchor=east]{$y_2$};
 \filldraw[red] (0.4,-0.5) node[anchor=east]{$tx+(1-t)y_i$};
  \end{tikzpicture}
  \caption{Typical picture for the convex combination of two v-points}
  \label{fig:cc-vpoints}
\end{figure}

  \textbf{Case 2.}
  $\lambda \in (0,1)$, i.e. $x$ is a convex combination
  of two distinct v-points $x_1$ and $x_2$.
  By definition there exist $y_1,y_2 \in S_X$ and $z_1,z_2 \in S_X$
  such that
  $\|y_i + x_i\| = 2$, $\|z_i + x_i\| = 2$ and $\|y_i + z_i\| < 2$.
  We may assume that $z_i = x$ for $i = 1,2$ (see Figure \ref{fig:cc-vpoints}).
  Let $\alpha > 0$ be such that $\max_i \|y_i + x\| < 2 - \alpha$.

  Find $x^* \in S_{X^*}$ such that $x^*(x) = 1$.
  Then $x^*(x_i) = 1$ and
  \begin{equation*}
    x^*(y_i) = x^*(x + y_i) - x^*(x) < 1 - \alpha,
  \end{equation*}
  so that $y_i \notin S(x^*,\alpha)$ for $i=1,2$.

  If $S$ is a subslice of $S(x^*,\alpha)$ and
  $z \in S \cap S_X$, then since $X$ is two-dimensional, $z$ is a convex
  combination of $x_1$ and $y_1$, of $x_2$ and $y_2$, or of $x_1$ and $x_2$.
  In either case we must have $x_1 \in S$ or $x_2 \in S$
  since $y_i \notin S$.
\end{proof}

\begin{lem}\label{lem:relative-Daugavet_direct_sum_3}
  Let $X$ be a two-dimensional Banach space and let $x \in S_X$.
  If there exist $x^* \in D(x)$ and $\alpha > 0$
  such that every subslice $S$ of $S(x^*,\alpha)$
  contains a point $y\neq x$ satisfying $\|x + y\| = 2$,
  then $x$ is a convex combination of two v-points.
\end{lem}

\begin{proof}
    Consider first the case when $x$ is an extreme point. Let
    $\varepsilon\in (0,1)$ be such that  $B_X \cap B(x,\varepsilon)
    \subset S(x^*,\alpha)$. If
    $B(x,\varepsilon) \cap \ext B_X = \set{x}$, then by
    Lemma~\ref{lem:char_vpt_sep_ext} $x$ is a v-point, as we
    wanted. Otherwise there exists $y_1\in B(x,\varepsilon) \cap \ext
    B_X$ such that $x\neq y_1$. Clearly $y_1\in S(x^*,\alpha)$, and
    since $y_1$ is an extreme point it is
    also a denting point. Let $ (S_i)$ be a sequence of slices such
    that $\diam S_i\rightarrow0$,
    $y_1\in S_i$, and $S_i\subset S(x^*,\alpha)$. Then by assumption
    for every $i\in\N$ we can find $z_i\in S_i$ such that $\|x+z_i\|=2$. Since
    $\|z_i - y_1\| \to
    0$,
    we also get $\|x+y_1\|=2$.

    Now let $\delta>0$ be such that $y_1\notin B(x,\delta)$. Again, if
    $B(x,\delta) \cap \ext B_X = \set{x}$, then by
    Lemma~\ref{lem:char_vpt_sep_ext} $x$ is a v-point, as we
    wanted. Otherwise there exists $y_2\in B(x,\delta) \cap \ext B_X$
    such that $x\neq y_2$. Similarly to before we can show that
    $\|x+y_2\|=2$. Let $y_1'$ and $y_2'$ be the midpoints of segments $[x,y_1]$ and $[x,y_2]$, respectively. Then we must have $\|y_1' + y_2'\| < 2$ since otherwise
    there exists $z^* \in S_{X^*}$ with $z^*(y_1' + y_2') = 2$, implying
    $z^*(y_i') = 1$ and thus $z^*(x)=z^*(y_i)=1$ for $i=1,2$, which would mean that $x$, $y_1$, and $y_2$ are all on the same line, since $X$ is two-dimensional. However, it is straightforward to check that three distinct extreme points can not all be on the same line. Thus $\|y_1+y_2\|<2$, and $x$ is a v-point. 
    
    If $x$ is not extreme, then $x$ can be written as
    a convex combination of two extreme points $x_1$ and $x_2$.
    Since $x^*(x) = 1$, we get $x^*(x_i) = 1$, and clearly $x_i$ also
    satisfies the condition of the lemma for the slice $S(x^*,\alpha)$ for $i=1,2$. 
    So from the above, it follows that $x_1$ and $x_2$ are both v-points.
\end{proof}

We now prove that the existence of a relative Daugavet-point in the
absolute sum of two Banach spaces forces the underlying absolute norm
to admit v-points.

\begin{thm}\label{thm:relative-Daugavet_direct_sum}
  Let $X$ and $Y$ be Banach spaces, $N$ be an absolute normalized
  norm on $\mathbb{R}^2$ and $(x,y)\in S_{X\oplus_N Y}$ be a
  relative Daugavet-point. Then $\big(\|x\|,\|y\|\big)$ is a convex
  combination of two v-points.
\end{thm}

\begin{proof}
  Let $Z:=X\oplus_NY$ and let us assume that $z:=(x,y)$ is a relative
  Daugavet-point in $S_Z$. Then there exist $\alpha>0$ and
  $z^*:=(x^*,y^*)\in D(z)$ such that for every subslice $S$ of $
  S(B_Z,z^*,\alpha)$ and $\varepsilon>0$, there exists $w\in S$ such
  that $\|z-w\|\ge 2-\varepsilon$. Let us consider
  $f=(\|x^*\|,\|y^*\|)\in (\R^2,N)^*$. Clearly, we have
  $$N^*(f)=1=f(\norm{x},\norm{y}).$$ 
  Let us fix a slice $S\big(B_{(\mathbb{R}^2,N)},(c,d),\beta\big)\subseteq
  S\big(B_{(\mathbb{R}^2,N)},f,\alpha\big)$.
  If $\norm{x^*}\neq 0$, then let $u^*=x^*/\|x^*\|$, otherwise let
  $u^*\in S_{X^*}$. Similarly if $\norm{y^*}\neq 0$, then let
  $v^*=y^*/\|y^*\|$, otherwise let $v^*\in S_{X^*}$. Then one can
  easily check that
  \[
    S\big(B_{Z},(cu^*,dv^*),\beta\big)
    \subseteq S(B_Z,z^*,\alpha).
  \]
  Indeed, if $(u,v)\in S\big(B_{Z},(cu^*,dv^*),\beta\big)$, then 
  $(u^*(u),v^*(v))\in S\big(B_{(\mathbb{R}^2,N)},(c,d),\beta\big)\subseteq
  S\big(B_{(\mathbb{R}^2,N)},f,\alpha\big)$, and thus 
  \[1-\alpha<\|x^*\|u^*(u)+\|y^*\|v^*(v)=x^*(u)+y^*(v)=z^*(u,v).\]
  
  So for every $n\in\mathbb{N}$ there exists
  $w_n:=(u_n,v_n)\in
  S\big(B_{Z},(cu^*,dv^*),\beta/2\big)$ such that
  $\|z-w_n\|\ge 2-1/n$. Then $(\norm{u_n},\norm{v_n})\in
  S\big(B_{(\mathbb{R}^2,N)},(c,d),\beta/2\big)$, and we have
  \[
  N\big(\|x\|+\|u_n\|,\|y\|+\|v_n\|\big)
    \ge
    N\big(\|x-u_n\|,\|y-v_n\|\big)
    = \|z-w_n\|\ge 2-\frac{1}{n}.
  \]
  Since $(\mathbb{R}^2,N)$ is norm compact, we may assume that
  $(\norm{u_n},\norm{v_n})_{n\geq1}$ converges in norm to some
  element $(a,b)\in S\big(B_{(\mathbb{R}^2,N)},(c,d),\beta\big)$.
  Then
  \[
    N\big(\|x\|+a,\|y\|+b\big)=2,
  \]
  and it follows from Lemma \ref{lem:relative-Daugavet_direct_sum_3}
  that $(\|x\|,\|y\|)$ is a convex combination of two v-points.
\end{proof}

In particular, combining this result with
Remark~\ref{rem:vertex-points}, we immediately get the following
corollary.

\begin{cor}\label{cor:strictly-convex_no_relative_Daugavet-point}
Let $X$ and $Y$ be Banach spaces, and $N$ be an absolute normalized
norm on $\R^2$. If $N$ is strictly convex, then $X\oplus_N Y$ does not
admit a relative Daugavet-point. In particular, $X\oplus_p Y$ never
admits relative Daugavet-points when $1<p<\infty$.
\end{cor}

Finally, we prove that the transfer of relative Daugavet-points in
absolute sums of Banach spaces can be characterized by the notion of
v-points. We will first introduce one additional lemma.

\begin{lem}\label{lem4:relative-Daugavet_direct_sum}
    Let $X$ and $Y$ be Banach spaces, and $Z:=X\oplus_NY$, where $N$
    is an absolute normalized norm on $\mathbb{R}^2$. If
    $S(B_Z,(u^*,v^*),\delta)\subseteq S(B_Z,(x^*,y^*),\beta)$ for
    $(u^*,v^*),(x^*,y^*)\in S_{Z^*}$ and $\delta,\beta>0$, then
    $S(B_{(\mathbb{R}^2,N)},(\norm{u^*},\norm{v^*}),\delta)\subseteq
    S( B_{(\mathbb{R}^2,N)},(\norm{x^*},\norm{y^*}),\beta)$.
\end{lem}

\begin{proof}
  Fix
  $(a,b)\in
  S(B_{(\mathbb{R}^2,N)},(\norm{u^*},\norm{v^*}),\delta)$. It suffices
  to consider two cases: $a,b\ge0$ and $a\ge0$, $b<0$. The other two
  cases are analogous. Assume first that $a,b\ge0$. Since
  $a\norm{u^*}+b\norm{v^*}>1-\delta$, we can take
  $u\in B_X$ and $v\in B_Y$ such that \[au^*(u)+bv^*(v)>1-\delta.\]
  Then
  $(au,bv)\in S(B_Z,(u^*,v^*),\delta)\subseteq S(B_Z,(x^*,y^*),\beta)$, and thus 
  \[a\norm{x^*}+b\norm{y^*}\geq
  ax^*(u)+by^*(v)>1-\beta,\]
  i.e., $(a,b)\in S( B_{(\mathbb{R}^2,N)},(\norm{x^*},\norm{y^*}),\beta)$.

  Now assume that $a\ge0$ and $b<0$. Since
  $a\norm{u^*}+b\norm{v^*}>1-\delta$, we can take $u\in B_X\setminus
  S_X$ such that
  \[
    au^*(u)+b\norm{v^*}>1-\delta.
  \]
  Then for every $v\in B_Y$ we have
  \[au^*(u)+bv^*(v)\ge au^*(u)+b\norm{v^*}>1-\delta\]
  and thus $(au,bv)\in S(B_Z,(u^*,v^*),\delta)\subseteq
  S(B_Z,(x^*,y^*),\beta)$. Note that $\|u\|<1$ and thus
  $a\norm{x^*}>x^*(au)$. Hence there exists $v\in B_Y$ such that
  \[a\norm{x^*}+b\norm{y^*}>
  ax^*(u)+by^*(v),\]
  and thus 
  \[a\norm{x^*}+b\norm{y^*}>1-\beta,\]
  i.e., $(a,b)\in S( B_{(\mathbb{R}^2,N)},(\norm{x^*},\norm{y^*}),\beta)$.
\end{proof}

\begin{thm}
  \label{thm:transfer_characterization_relative_Daugavet-points}
  Let $X$ and $Y$ be Banach spaces, $x\in S_X$ and $y\in S_Y$ be
  relative Daugavet-points, $N$ be an absolute normalized norm on
  $\mathbb{R}^2$ and $a,b\ge 0$ be such that $N(a,b)=1$. Then
  $(ax,by)$ is a relative Daugavet-point of $X\oplus_NY$ if and only
  if $(a,b)$ is a convex combination of two v-points.
\end{thm}

\begin{proof}
  The ``only if'' part is Theorem \ref{thm:relative-Daugavet_direct_sum}.

  Assume $(a,b)$ is a convex combination of two v-points
  $(a_1,b_1)$ and $(a_2,b_2)$. Let $Z=X\oplus_NY$ and $z := (ax,by)$.
  By Lemma~\ref{lem:relative-Daugavet_direct_sum_1and2} there exists a slice
  $S(B_{(\mathbb{R}^2,N)},(c,d),\alpha)$ such that $N^*(c,d)=1$ and
  $ac+bd=1$, and such
  that for every subslice
  $S$ of $S(B_{(\mathbb{R}^2,N)},(c,d),\alpha)$, either
  $(a_1,b_1)\in S$ or $(a_2,b_2)\in S$.
  By making $\alpha > 0$ smaller, if necessary, we may assume that
  $x$ is a Daugavet-point relative to $S(B_X,x^*,\alpha)$,
  $x^* \in D(x)$, and
  $y$ is a Daugavet-point relative to $S(B_Y,y^*,\alpha)$,
  $y^* \in D(y)$.

  Since $(a,b)$ is a convex combination of $(a_1,b_1)$ and $(a_2,b_2)$
  we have
  \begin{equation}\label{eq1_relative-Daugavet_direct_sum}
    a_1c+b_1d=a_2c+b_2d=ac+bd=1.
  \end{equation}
  Note that if $a_1<0$, then $a_2>0$ and $b_1=b_2=b=1$,
  but then $a_1c = ac$, which is only possible when $c=0$.
  Similarly we can show that if $a_2<0$, then $c=0$,
  and if $b_1<0$ or $b_2<0$, then $d=0$.
  Thus 
  $\min \big(\{a_1c,a_2c,b_1d,b_2d\}\setminus\{0\}\big) > 0.$ 

  Let $z^*:=(cx^*,dy^*)$.
  Clearly $z^*\in S_{Z^*}$ and $z^*(z)=acx^*(x)+bdy^*(y)=1$.
  Let $\beta>0$ be such that
  \begin{equation*}
    \beta < \alpha \min \big(\{a_1c,a_2c,b_1d,b_2d\}\setminus\{0\}\big).
  \end{equation*}
  Our goal is to show that $z = (ax,by)$ is a Daugavet-point relative to
  the slice $S(B_Z,z^*,\beta)$.
  Fix $S(B_Z,w^*,\delta)\subseteq S(B_Z,z^*,\beta)$ and
  $\varepsilon>0$. Write $w^*:=(u^*,v^*)$. By Lemma
  \ref{lem4:relative-Daugavet_direct_sum} we get
  \[
    S(B_{(\mathbb{R}^2,N)},(\norm{u^*},\norm{v^*}),\delta)\subseteq S(
    B_{(\mathbb{R}^2,N)},(c,d),\alpha),
  \] 
  and thus either
  $a_1\|u^*\|+b_1\|v^*\|>1-\delta$ or
  $a_2\|u^*\|+b_2\|v^*\|>1-\delta$.
  Without loss of generality assume
  $a_1\|u^*\|+b_1\|v^*\|>1-\delta$.

  Let $\gamma>0$ be such that
  $a_1\|u^*\|+b_1\|v^*\|>1-\delta+2\gamma$.
  We will first show that if $u^*\neq 0$ and $a_1c\neq0$, then
  $S(B_X,u^*/\|u^*\|,\gamma)\subseteq S(B_X,x^*,\alpha)$.
  Assume $u^*\neq 0$, $a_1c\neq0$ and fix $u\in
  S(B_X,u^*/\|u^*\|,\gamma)$. Then
  \[
    u^*(a_1u)>a_1\|u^*\|-a_1\|u^*\|\gamma\ge a_1\|u^*\|-\gamma.
  \]
  Choose $v\in B_Y$ such that 
  \[
    u^*(a_1u)+v^*(b_1v)
    >a_1\|u^*\|-2\gamma+b_1\|v^*\|
    >1-\delta.
  \]
  Then $w^*(a_1u,b_1v)>1-\delta$. By combining $S(B_Z,w^*,\delta)\subseteq
  S(B_Z,z^*,\beta)$, $b_1d\ge0$, and
  \eqref{eq1_relative-Daugavet_direct_sum} we get
  \[
    1-\beta
    <cx^*(a_1u)+dy^*(b_1v)
    \le a_1cx^*(u)+b_1d
    = a_1cx^*(u)+1-a_1c.
  \]
  Then $x^*(u)>1-\beta/(a_1c)>1-\alpha$ and we have shown that
  $S(B_X,u^*/\|u^*\|,\gamma)\subseteq S(B_X,x^*,\alpha)$. If $v^*\neq
  0$ and $b_1d\neq0$, then the inclusion
  $S(B_Y,v^*/\|v^*\|,\gamma)\subseteq S(B_Y,y^*,\alpha)$
  can be proved
  analogously. Now consider three cases.

  \textbf{Case 1.} Assume $u^*\neq 0$, $a_1c\neq0$, $v^*\neq 0$ and
  $b_1d\neq0$. Then there exist $u\in S(B_X,u^*/\|u^*\|,\gamma)$ such
  that $\|x-u\|\ge 2-\varepsilon/2$ and $v\in S(B_Y,v^*/\|v^*\|,\gamma)$
  such that $\|y-v\|\ge 2-\varepsilon/2$.
  Let $w:=(a_1u,b_1v)$. Then using
  $a_1\|u^*\| + b_1\|v^*\| \le N(a_1,b_1)N^*(\|u^*\|,\|v^*\|) = 1$,
  we get
  \[
    w^*(w)
    =a_1u^*(u)+b_1v^*(v)
    >a_1\|u^*\|(1-\gamma) + b_1\|v^*\|(1-\gamma)
    >a_1\|u^*\| + b_1\|v^*\| - \gamma
    > 1-\delta,
  \]
  giving us $w\in S(B_Z,w^*,\delta)$.
  Lastly notice that 
  \[
    \|ax-a_1u\|
    \ge \max\{a,a_1\}\|x-u\|-|a-a_1|
    \ge \max\{a,a_1\}(2-\frac{\varepsilon}{2})-|a-a_1|
    \ge a+a_1-\frac{\varepsilon}{2}
  \]
  and analogously $\|by-b_1v\|\ge b+b_1-\varepsilon/2$, giving us
  \begin{align*}
    \|z-w\|
    &= N\big(\|ax-a_1u\|,\|by-b_1v\|\big)\\
    &\ge N(a+a_1-\varepsilon/2,b+b_1-\varepsilon/2)\\
    &\ge N(a+a_1,b+b_1)-N(\varepsilon/2,\varepsilon/2)\\
    &\ge2-\varepsilon
  \end{align*}
  as $N(a+a_1,b+b_1)=2$ (see Figure~\ref{fig:cc-vpoints}) and as every
  absolute normalized norm $N$ on $\R^2$ satisfies
  $\norm{\cdot}_\infty\leq N \leq \norm{\cdot}_1$.

  \textbf{Case 2.}
  Assume $u^*= 0$ or $a_1=0$
  (the case $v^*= 0$ or $b_1=0$ is analogous).
  Then $b_1\|v^*\|>1-\delta+\gamma>0$ and thus $d>1-\alpha>0$,
  since
  $(0,b_1) \in S(B_{(\mathbb{R}^2,N)},(\norm{u^*},\norm{v^*}),\delta)
  \subseteq S( B_{(\mathbb{R}^2,N)},(c,d),\alpha)$.
  Therefore $v^* \neq 0$ and $b_1d \neq 0$ and
  since $S(B_Y,v^*/\|v^*\|,\gamma) \subset S(B_Y,y^*,\alpha)$
  there exists $v \in S(B_Y,v^*/\|v^*\|,\gamma)$ such that
  $\|y-v\| \ge 2-\varepsilon/2$.

  Let $w:=(-a_1x,b_1v)$. Then
  \[
    w^*(w)
    =-a_1u^*(x)+b_1v^*(v)=b_1v^*(v)
    >b_1\|v^*\|-b_1\|v^*\|\gamma
    \ge b_1\|v^*\|-\gamma
    > 1-\delta,
  \]
  giving us $w\in S(B_Z,w^*,\delta)$.
  As in Case~1 we get $\|by-b_1v\|\ge b+b_1-\varepsilon/2$ and thus
  \[
    \|z-w\|
    \ge N(a+a_1,b+b_1-\varepsilon/2)
    \ge N(a+a_1,b+b_1)-N(0,\varepsilon/2)
    \ge2-\varepsilon.
  \]

  \textbf{Case 3.}
  Assume $u^*\neq 0$, $a_1\neq0$, $v^*\neq 0$, $b_1\neq0$, and $c=0$
  (the case $d=0$ is analogous).
  Then $bd=1$ and thus also $b_1=b_2=1$.
  Since $v^*\neq 0$ and $b_1d \neq 0$ there exists $v\in
  S(B_Y,v^*/\|v^*\|,\gamma)$ such that $\|y-v\|\ge
  2-\varepsilon/2$. Take any $u\in S(B_X,u^*/\|u^*\|,\gamma)$  and let
  $w:=(a_1u,b_1v)$. Then
  \[
    w^*(w)=a_1u^*(u)+b_1v^*(v)
    >a_1\|u^*\| + b_1\|v^*\|-\gamma
    > 1-\delta,
  \]
  giving us $w\in S(B_Z,w^*,\delta)$.
  As in Case~1 we get $\|by-b_1v\|\ge b+b_1-\varepsilon/2$ and thus
  \[
    \|z-w\|
    \ge N(0,b+b_1-\varepsilon/2)
    = N(0,2-\varepsilon/2)
    \ge2-\varepsilon.
  \]
Therefore $z$ is a relative Daugavet-point.
\end{proof}

Again, combining this result with Remark~\ref{rem:vertex-points}, we
immediately get the following corollary.

\begin{cor}\label{cor:polyhedral_stability_relative_Daugavet-points}
  Let $X$ and $Y$ be Banach spaces, and $N$ be an absolute normalized
  norm on $\R^2$ for which $(\R^2,N)$ is polyhedral. If $x\in S_X$ and
  $y\in S_Y$ are relative Daugavet-points, then $(ax,by)$ is a
  relative Daugavet-point in $X\oplus_N Y$ for every $(a,b)\in\R^2$
  with $N(a,b)=1$. Furthermore, if $a=0$ (respectively $b=0$), then
  the assumption $x\in S_X$ is a relative Daugavet-point (respectively
  $y\in S_Y$ is a relative Daugavet-point) can be dropped.
\end{cor}

Let us end the section by collecting specific applications of the
previous results. We first introduce the following definition.

\begin{defn}
    \label{defn:relative_Daugavet-property}
    Let $X$ be a Banach space. We say that $X$ has the
    \emph{relative Daugavet property} if every point of $S_X$ is a
    relative Daugavet-point.
\end{defn}

Note that Corollary~\ref{cor:polyhedral_stability_relative_Daugavet-points}, Theorem~\ref{thm:relative-Daugavet_direct_sum} and
Lemma~\ref{lem:char_vpt_sep_ext} immediately imply the following
stability result for the relative Daugavet property.

\begin{cor}\label{cor:polyhedral_transfer_relative_Daugavet-property}
Let $X$ and $Y$ be Banach spaces, and $N$ be an absolute normalized
norm on $\R^2$. If $X$ and $Y$ have the relative Daugavet property,
then $X\oplus_N Y$ has the relative Daugavet property if and only if
$(\R^2,N)$ is polyhedral.
\end{cor}

From the results of Subsection~\ref{subsec:relative_Daugavet-points}, the
Daugavet property implies the relative Daugavet property, and the
relative Daugavet property implies the DLD2P. We now put our results
together and prove what we claimed at the start of the subsection,
that is that this property lies strictly between the Daugavet property
and the DLD2P.

\begin{cor}\label{cor:relative-Daugavet_no_daugavet-point}
  There exists a Banach space with the relative Daugavet property
  that fails to contain Daugavet-points.
\end{cor}

\begin{proof}
  Let $X$ and $Y$ be Banach spaces with the relative Daugavet property
  (e.g. spaces with the Daugavet property).
  Let $N$ be an absolute norm on $\R^2$
  which has property $(\alpha)$ in the sense of \cite{AHLP} and for
  which $(\R^2,N)$ is polyhedral.
  Figure~\ref{figure:polyhedral+alpha_absolute-norm}
  gives one example of such a norm.
  The space $X\oplus_N Y$ has the relative Daugavet property by
  Corollary~\ref{cor:polyhedral_transfer_relative_Daugavet-property}.
  However, by \cite[Proposition~4.6]{AHLP}
  $X \oplus_N Y$ fails to contain Daugavet-points.
\end{proof}

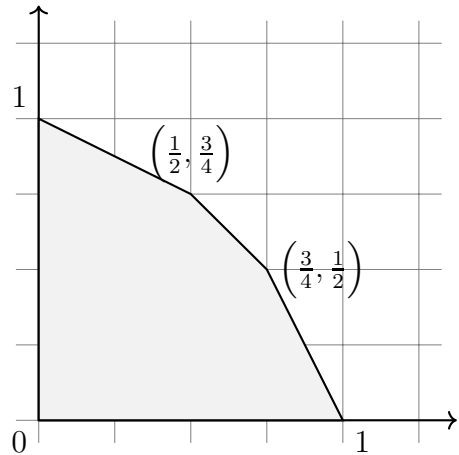
\begin{figure}[!ht]
	\begin{center}
		\begin{tikzpicture}[thick, scale=1]
		\draw[step=1cm,gray,very thin] (-0.3,-0.3) grid (5.3,5.3);
		\draw[->] (0,0) -- (5.5,0);
		\draw[->] (0,0) -- (0,5.5);
		\draw[thick, fill=gray!10] (0,0) -- (0,4) -- (2,3) -- (3,2) -- (4,0) -- cycle;
		\node [below right] at (4,0) {1};
		\node [above left] at (0,4) {1};
		\node [below left] at (0,0) {$0$};
		\node [above] at (2,3) {$\Big(\frac{1}{2},\frac{3}{4}\Big)$};
		\node [right] at (3,2) {$\Big(\frac{3}{4},\frac{1}{2}\Big)$};
		\end{tikzpicture}
	\end{center}
\caption{Positive cone of an absolute polyhedral norm with property
  $(\alpha)$}\label{figure:polyhedral+alpha_absolute-norm}
\end{figure}

\begin{cor}
  There exists a Banach space with the DLD2P
  that does not contain any relative Daugavet-point. 
\end{cor}

\begin{proof}
   Let $X$ and $Y$ be Banach spaces with the DLD2P. Then $X\oplus_2
  Y$ fails to contain relative Daugavet-points by
  Corollary~\ref{cor:strictly-convex_no_relative_Daugavet-point}. However,
 every point in the unit sphere of $X\oplus_2  Y$ is a
  $\Delta$-point (see Section~4 in \cite{AHLP}), so the space has the DLD2P.
\end{proof}

Let us end the section with a few remarks and questions.
The space from \cite[Theorem~2.11]{AHNTT} is midpoint locally
uniformly rotund and satisfies the DLD2P.
As a consequence, it also satisfies the DD2P and the restricted
diametral strong diameter 2 property (restricted DSD2P in the sense of
\cite{MPRZ}), but as an $\ell_2$-sum of Banach spaces it fails to
contain any relative Daugavet-point.
On the other hand, it is well known that the Daugavet property implies
all the diameter 2 and diametral diameter 2 properties, so it make
sense to ask what happens with the relative Daugavet property.

\begin{quest}
    Does the relative Daugavet property imply something stronger than
    the DLD2P? More precisely, does it imply the D2P, DD2P, SD2P or
    restricted DSD2P?
\end{quest}

Of course the relative Daugavet property does not imply the DSD2P, as
the latter is equivalent to the Daugavet property \cite{Kadets21}.

Finally, it is natural in view of the results from
\cite[Section~3]{HPV} to ask if any transfer result is available in
the other direction.

\begin{quest}
Do relative Daugavet-points go down from absolute sums to summands?
What about the relative Daugavet property?
\end{quest}


\section{Delta- and Daugavet-points in subspaces of
  \texorpdfstring{$L_1$}{L1}}
\label{sec:subL1}
There are large classes of Lipschitz-free spaces for which all extreme
points of the unit ball are denting points. For example this is true
for Lipschitz-free spaces $\lipfree{M}$ when $M$ is a compact metric space
combining \cite[Theorem~1]{Aliaga_2022} and \cite[Theorem~2.4]{GLPPRZ} or
when $M$ is a subset of an $\mathbb{R}$-tree combining
\cite[Theorem~4.2]{Godard}, \cite[Theorem~4.4]{APP}, and
\cite[Theorem~2.4]{GLPPRZ}. From Theorem
\ref{thm:KMP_denting} it therefore follows
that such spaces fail to contain Daugavet-points whenever they have
the KMP. But even more can be said. We will prove
in Subsection \ref{subsec:subL1-free} that in free spaces over subsets
of $\mathbb{R}$-trees (see below for a definition)
Daugavet- and $\Delta$-points are the same, so Theorem
\ref{thm:KMP_denting} takes a stronger form in these
spaces. 

Recall that it was proved in \cite[Theorem~3.1]{AHLP} that Daugavet-
and $\Delta$-points coincide in $L_1(\mu)$ spaces when $\mu$ is a
$\sigma$-finite measure. The result was extended to
arbitrary measures in \cite{MPRZ}, and it was actually shown in
\cite[Corollary~4.11]{MPRZ} that these points
also coincide with their super versions in this context.
Furthermore, recall that free spaces over subsets of $\mathbb{R}$-trees
are precisely the Lipschitz-free spaces which embed isometrically into
an $L_1(\mu)$ space by \cite[Theorem~4.2]{Godard}. Thus it
naturally opens up the question whether $\Delta$- and
Daugavet-points are the same in all subspaces of $L_1[0,1]$.
In Subsection \ref{subsec:subL1-reflexive} we will prove that none of
the reflexive subspaces of $L_1[0,1]$ contain $\mathfrak{D}$-points, hence 
that the answer to the previous question is trivially positive for
these subspaces. Nevertheless, in Subsection
\ref{subsec:subL1_relative-Daugavet_no_daugavet} we will construct a
subspace of $L_1[0,1]$ with a $\Delta$-point that is not a relative
Daugavet-point.

\subsection{Free spaces over subsets of
  \texorpdfstring{$\mathbb{R}$}{R}-trees}
\label{subsec:subL1-free}
As already announced, this subsection is dedicated to showing that
Daugavet-points and $\Delta$-points are the same in Lipschitz-free spaces over
subsets of $\mathbb{R}$-trees.

Recall that an \emph{$\mathbb{R}$-tree} is an arc-connected metric
space $(T, d)$ with the property that there is a unique arc connecting
any pair of points $x \neq y \in T$ which is moreover isometric to the real
segment $[0, d(x,y)] \subset \mathbb{R}.$ Such an arc, denoted
$[x,y]$, is called a \emph{segment} of $T$ and it is immediate that
it coincides with the metric segment
\[
    [x, y]
    = \{p \in T: d(x,p) + d(p, y) = d(x, y)\}.
\]

Let $T$ be an $\mathbb{R}$-tree. Then for every $x,y,p\in T$, there
exists a unique element $z_p\in [x,y]$ such that $d(p,z_p)=\min_{q\in
  [x,y]}d(p,q)$. Let $\Ymap_{xy}\colon T\rightarrow T$ be defined by
\begin{equation*}
  \Ymap_{xy}p := z_p.
\end{equation*}
Then $\Ymap_{xy}$ is a 1-Lipschitz retraction of $T$ onto $[x,y]$.
Furthermore, if for some $p,q\in T$ we have $\Ymap_{xy}p\neq \Ymap_{xy}q$,
then \begin{align}
\label{eq:1}
    [p,q] =
[p,\Ymap_{xy}p] \cup [\Ymap_{xy}p,\Ymap_{xy}q] \cup [\Ymap_{xy}q,q],
\end{align}
so that $\Ymap_{xy}p,\Ymap_{xy}q\in [p,q]$ and thus
\begin{align}
  \label{eq:2}
  d(p,q)
  =d(p,\Ymap_{xy}p)+d(\Ymap_{xy}p,\Ymap_{xy}q)+d(\Ymap_{xy}q,q).
\end{align}
Arguing by contradiction we get the following fact from \eqref{eq:2}
\begin{equation}
  \label{eq:3}
  \text{If}\ \Ymap_{xy}p = \Ymap_{xy}q,
  \ \text{then}\ \Ymap_{xy}r = \Ymap_{xy}p
  \ \text{for all}\ r\in [p,q].
\end{equation}

When $M$ is a (complete) subset of an $\R$-tree, there is a unique
smallest $\R$-tree $T$ containing $M$. For $x,y\in M$, the notation
$[x,y]\subset M$ then means that the corresponding segment in $T$ is
in fact contained in $M$ --- equivalently, that $x$ and $y$ are
connected by a geodesic in $M$. In that case, $\Ymap_{xy}$ is still a
well-defined 1-Lipschitz retraction in $M$.

We start by proving that the notions of Daugavet- and $\Delta$-points
coincide for molecules in Lipschitz-free spaces over subsets of
$\mathbb{R}$-trees, and provide a specific metric characterization of
these molecules.

\begin{prop}\label{prop:Rtree-mol-warmup}
  Let $M$ be a complete subset of an $\mathbb{R}$-tree $T$,
  and let $x\neq y\in M$.
  The following assertions are equivalent.
  \begin{enumerate}
  \item\label{item:Daugavet-molecule_R-tree}
    The molecule $m_{xy}$ is a super Daugavet-point in $\lipfree{M}$;
  \item\label{item:Delta-molecule_R-tree}
    The molecule $m_{xy}$ is a $\Delta$-point in $\lipfree{M}$;
  \item\label{item:metric-prop_R-tree} $[x,y]\subset M$.
  \end{enumerate}
\end{prop}

\begin{proof}
  \ref{item:Daugavet-molecule_R-tree}$\Rightarrow$\ref{item:Delta-molecule_R-tree}
  is obvious, and
  \ref{item:Delta-molecule_R-tree}$\Rightarrow$\ref{item:metric-prop_R-tree}
  follows from \cite[Theorem~6.7]{AALMPPV} and completeness of $M$ since if
  $\varepsilon > 0$ and $z \in [x,y]$ with $r = d(x,z)$, then we have
  for all $p \in B(x,r+\varepsilon) \cap B(y,d(x,y)-r + \varepsilon)$ that
  \begin{equation*}
    d(z,p) \le d(z,\Ymap_{xy}p) + d(\Ymap_{xy}p,p) \le 2\varepsilon.
  \end{equation*}

  Let us show
  \ref{item:metric-prop_R-tree}$\Rightarrow$\ref{item:Daugavet-molecule_R-tree}.
  If $[x,y]$ is a geodesic in $M$,
  then $m_{xy}$ is a Daugavet-point in $\lipfree{M}$ by e.g.
  \cite[Theorem~2.1]{VeeorgStudia}. Let us see why:
  Let $\mu\in\dent B_{\lipfree{M}}$.
  Then $\mu=m_{uv}$ where $[u,v]=\set{u,v}$ by
  \cite[Corollary~4.5]{APP} (or simply
  \cite[Corollary~3.44]{Weaver2}).
  Thus we have $\Ymap_{xy}u=\Ymap_{xy}v$, since otherwise $[u,v]$
  contains the
  geodesic $[\Ymap_{xy}u,\Ymap_{xy}v]$ by \eqref{eq:1}.
  Writing $p=\Ymap_{xy}u=\Ymap_{xy}v$, we get
  \begin{equation*}
    d(x,u)+d(y,v)=d(x,p)+d(p,u)+d(y,p)+d(p,v)\geq d(x,y)+d(u,v)
  \end{equation*}
  which implies $\norm{m_{xy}-m_{uv}}=2$ by
  \cite[Lemma~1.2]{VeeorgStudia}.
  Thus $m_{xy}$ is a Daugavet-point in $\lipfree{M}$ by
  \cite[Theorem~2.1]{VeeorgStudia}.
  But we can say a bit more.
  Indeed, we can show that $\lipfree{[x,y]}$ is an L-summand in
  $\lipfree{M}$, where we fix e.g. $x$ as the base point.
  Then as $\lipfree{[x,y]}$ is isometrically isomorphic to $L_1[0,1]$,
  and as this space has the Daugavet property,
  every point on the unit sphere of $\lipfree{[x,y]}$ is
  a super Daugavet-point.
  In particular, $m_{xy}$ is a super Daugavet-point in
  $\lipfree{[x,y]}$ and since these points transfer through $\ell_1$-sums
  by \cite[Remark~3.28]{MPRZ},
  we get that $m_{xy}$ is a super Daugavet-point in $\lipfree{M}$.

So it only remains to prove that $\lipfree{[x,y]}$ is an L-summand in
$\lipfree{M}$.
Clearly $\Ymap_{xy}$ is a 1-Lipschitz retraction of $M$ onto $[x,y]$, so
its linearization $\widehat{\Ymap}_{xy}:\lipfree{M}\to\lipfree{[x,y]}$ is
a norm-one projection.
To see that it is an L-projection, it suffices to check that
$\norm{\mu}=\norm{\widehat{\Ymap}_{xy}\mu}+\norm{\mu-\widehat{\Ymap}_{xy}\mu}$
when $\mu\in\lipfree{M}$ has finite support. So write
\begin{equation*}
  \mu = \sum_{i=1}^n a_i\delta(p_i) + \sum_{j=1}^m b_j\delta(q_j)
\end{equation*}
where $p_i\in [x,y]$, $q_j\notin [x,y]$.
Then $\widehat{\Ymap}_{xy}\mu
= \sum_i a_i\delta(p_i) + \sum_j b_j\delta(\Ymap_{xy}q_j)$
and
$\mu-\widehat{\Ymap}_{xy}\mu=\sum_j b_j(\delta(q_j)-\delta(\Ymap_{xy}q_j))$.
Choose $f,g\in B_{\Lip_0(M)}$ such that
$f\left(\widehat{\Ymap}_{xy}\mu\right)=\norm{\widehat{\Ymap}_{xy}\mu}$ and
$g\left(\mu-\widehat{\Ymap}_{xy}\mu\right)=\norm{\mu-\widehat{\Ymap}_{xy}\mu}$.
Now define a function $h$ by
\begin{equation*}
  h(p) := f(\Ymap_{xy}p) + g(p) - g(\Ymap_{xy}p)
\end{equation*}
for $p\in M$.
Notice that, for $p,q\in M$, we either have $\Ymap_{xy}p=\Ymap_{xy}q$
and thus $h(p)-h(q)=g(p)-g(q)\leq d(p,q)$, or $\Ymap_{xy}p\neq \Ymap_{xy}q$ and by \eqref{eq:3}
\begin{align*}
  h(p)-h(q)
  &=
  g(p) - g(\Ymap_{xy}p) + f(\Ymap_{xy}p) - f(\Ymap_{xy}q) + g(\Ymap_{xy}q) - g(q) \\
  &\leq
  d(p,\Ymap_{xy}p)+d(\Ymap_{xy}p,\Ymap_{xy}q)+d(\Ymap_{xy}q,q)
  = d(p,q).
\end{align*}
Thus $h\in B_{\Lip_0(M)}$, and
\begin{align*}
  \norm{\mu} \geq \duality{\mu,h}
  &=
  \sum_{i=1}^n a_i \bigl(f(\Ymap_{xy}p_i)+g(p_i)-g(\Ymap_{xy}p_i)\bigr)
  + \sum_{j=1}^m b_j\bigl(f(\Ymap_{xy}q_j)+g(q_j)-g(\Ymap_{xy}q_j)\bigr) \\
  &=
  \sum_{i=1}^n a_if(p_i) + \sum_{j=1}^m b_jf(\Ymap_{xy}q_j)
  + \sum_{j=1}^m b_j\bigl(g(q_j)-g(\Ymap_{xy}q_j)\bigr) \\
  &=
  f\bigl(\widehat{\Ymap}_{xy}\mu\bigr)
  + g\bigl(\mu-\widehat{\Ymap}_{xy}\mu\bigr)
  =
  \norm{\widehat{\Ymap}_{xy}\mu} + \norm{\mu-\widehat{\Ymap}_{xy}\mu}.
  \qedhere
\end{align*}
\end{proof}

Next we show that the characterization of $\Delta$-molecules from
\cite[Theorem~4.7]{JungRueda} holds for arbitrary elements of
Lipschitz-free spaces over subsets of $\mathbb{R}$-trees. Note that
this characterization is already known to hold for elements with
finite support in arbitrary Lipschitz-free spaces by
\cite[Theorem~4.4]{VeeorgStudia} and for all elements in
Lipschitz-free spaces over proper metric spaces by
\cite[Proposition~3.2]{VeeorgFunc}. We will actually first prove the
characterization for certain elements of general Lipschitz-free
spaces, including elements with bounded support in Lipschitz-free
spaces over subsets of $\mathbb{R}$-trees. To do so, we need the
following lemma, which is almost identical to
\cite[Lemma~4.3]{VeeorgStudia}. The only difference is that we allow
$\alpha > 0$ to be arbitrary and reach a slightly weaker conclusion.

\begin{lem}{\cite[Lemma~4.3]{VeeorgStudia}}\label{Lem:f_mu}
    Let $M$ be a metric space and $\mu\in S_{\lipfree{M}}$. If $\mu$
    can be written as a convex combination of molecules $m_{x_iy_i}$
    with $i\in\{1,\dots,n\}$, then there exists $f_\mu\in
    S_{\Lip_0(M)}$ such that the following hold:
    \begin{enumerate}[label={(\arabic*)}]
        \item $f_\mu(\mu) =1$;
        \item For every $u,v\in M$ and $\alpha>0$ with $u\neq v$ and
          $m_{uv}\in S(f_{\mu},\alpha)$ there exist $i,j\in
          \{1,\ldots,n\}$ with $x_{i}\neq y_{j}$ such that
        \begin{equation*}
            (1 - \alpha)\max\{d(x_{i},v) + d(y_{j},v),d(x_{i},u) +
            d(y_{j},u)\}< d(x_{i},y_{j}).
        \end{equation*}
    \end{enumerate}
\end{lem}

Recall that for a bounded subset $A$ of a Banach space $X$, the
\emph{Kuratowski measure of non-compactness} $\alpha(A)$ is the
infimum of such $\varepsilon>0$ that $A$ can be covered with a finite
number of sets with diameter less than $\varepsilon$.

\begin{prop}\label{prop:delta_small_molecule_characterization}
    Let $M$ be a metric space
    and $\mu\in S_{\mathcal{F}(M)}$. Assume that
    \begin{equation*}
        \lim_{\delta\to 0} \sup_{x,y\in\supp(\mu)\cup\set{0}}
        \alpha\left(\set{ p\in M \,:\, d(p,x)+d(p,y)<(1+\delta)d(x,y)
          }\right)=0.
    \end{equation*}
    Then $\mu$ is a $\Delta$-point if and only if for every slice $S$
    of $B_{\lipfree{M}}$ with $\mu\in S$ and for every $\varepsilon>0$
    there exists $m_{uv}\in S$ such that $d(u,v)<\varepsilon$.
\end{prop}

\begin{proof}
    One implication can be derived directly from {\cite[Theorem~2.6]{JungRueda}}.
    
    Now fix $\mu\in S_{\mathcal{F}(M)}$
    and additionally assume it is a $\Delta$-point. Fix a slice
    $S(f,\delta)$ of $B_{\lipfree{M}}$ with $\mu\in S(f,\delta)$ and
    $\varepsilon>0$. By Lemma~\ref{lem:diminution_slices} we may
    assume that $\delta<1/2$ and
    \begin{equation}\label{eq_prop:delta_small_mol}
        \sup_{x,y\in\supp(\mu)\cup\set{0}} \alpha\left(\set{ p\in M
            \,:\, d(p,x)+d(p,y)<(1+2\delta)d(x,y)
          }\right)<\varepsilon(1-\delta).
    \end{equation}
    Let $\gamma>0$ be such that $\mu\in S(f,\delta-\gamma)$. There
    exist $n\in\mathbb{N}$,
    $x_1,\ldots,x_n,y_1,\ldots,y_n\in\supp(\mu)\cup\set{0}$ and
    $\lambda_1,\ldots,\lambda_n>0$ with
    $\sum^{n}_{i=1}\lambda_i=1$ such that $\sum^{n}_{i=1}\lambda_i
    m_{x_iy_i}\in S_{\lipfree{M}}$ and
    \begin{equation*}
        \norm{\mu-\sum^{n}_{i=1}\lambda_i m_{x_iy_i}}<\gamma.
    \end{equation*}
    Let $\nu=\sum^{n}_{i=1}\lambda_i m_{x_iy_i}$ and let $f_{\nu}$ be
    the function from Lemma \ref{Lem:f_mu}. Then
    $(f+f_\nu)(\mu)>2-\delta$. Since $\mu$ is a $\Delta$-point, then
    by {\cite[Proposition~3.1]{VeeorgFunc}} there exists a sequence
    $(m_{u_iv_i})$ such that
    $(f+f_\nu)(m_{u_iv_i})>2-\delta$ and
    \[
      \|m_{u_iv_i}-m_{u_jv_j}\|
      \ge 2-\delta
    \]
    for all $i,j\in \mathbb{N}$ with $i\neq j$. Thus
    $(m_{u_iv_i})\subseteq S(f,\delta)$ and $(m_{u_iv_i})\subseteq
    S(f_\nu,\delta)$. Therefore there exist $i,j\in \{1,\ldots,n\}$
    such that
    \begin{equation*}
        (1 - \delta)\max\{d(x_{i},v_k) + d(y_{j},v_k),d(x_{i},u_k) +
        d(y_{j},u_k)\}< d(x_{i},y_{j})
    \end{equation*}
    for infinitely many $k\in\mathbb{N}$.
    Since $\delta<1/2$, we have $(1-\delta)^{-1} \le 1 + 2\delta$.
    Hence by \eqref{eq_prop:delta_small_mol} there exist distinct
    $k,l\in\mathbb{N}$ such that
    \begin{equation*}
        \max\{d(u_k,u_l),d(v_k,v_l)\}<\varepsilon(1-\delta).
    \end{equation*}
    Then 
    \begin{equation*}
        (1-\delta)(d(u_k,v_k)+d(u_l,v_l))\le d(u_k,u_l)+d(v_k,v_l)<2\varepsilon(1-\delta)
    \end{equation*}
    where the left inequality follows from \cite[Lemma 1.2]{VeeorgStudia}.
    Thus either $d(u_k,v_k)<\varepsilon$ or $d(u_l,v_l)<\varepsilon$. 
\end{proof}

The above proposition can be applied, for example, to elements of
Lipschitz-free spaces over metric spaces that are bounded subsets of
$\ell_1$. We will now prove that we can use similar ideas to
generalize this to arbitrary subsets of $\R$-trees. For this purpose,
we first need to construct a Lipschitz function that will play the
role of the functional $f_\nu$ from the proof of
Proposition~\ref{prop:delta_small_molecule_characterization}.

\begin{lem}\label{Lem:function_R_tree}
   Let $M$ be a complete subset of an $\mathbb{R}$-tree $T$ and
   $\mu\in S_{\lipfree{M}} $. If $\mu$ can be written as a convex
   combination of molecules $m_{x_iy_i}$ with $i\in\{1,\dots,n\}$,
   then there exists $g_\mu\in S_{\Lip_0(M)}$ such that the following
   hold:
    \begin{enumerate}[label={(\arabic*)}]
        \item\label{item1_Lem:function_R_tree} $g_\mu(\mu) =1$;
        \item\label{item2_Lem:function_R_tree} For every $u,v\in M$
          and $\alpha>0$ with $u\neq v$ and $m_{uv}\in
          S(g_{\mu},\alpha)$ there exist $i,j\in \{1,\ldots,n\}$ with
          $x_{i}\neq y_{j}$ such that
        \begin{equation*}
            (1-\alpha)d(u,v)<
            d(\Ymap_{x_iy_j}u, \Ymap_{x_iy_j}v).
        \end{equation*}
    \end{enumerate}
\end{lem}

\begin{proof}
  Let $f\in S_{\Lip_0(M)}$ be such that $f(\mu) = 1$.
  Then $f(m_{x_iy_i})=1$ for every $i\in\{1,\ldots,n\}$.
  Define $g_{\mu}\colon M\rightarrow\mathbb{R}$ by
  \begin{equation*}
    g_{\mu}(p)
    :=
    \max_{i\in\{1,\ldots,n\}}
    \big( f(x_i) - \max_{j\in\{1,\ldots,n\}} d(x_i,\Ymap_{x_iy_j}p) \big).
  \end{equation*}
  For every $p,q\in M$ we have
  \begin{equation*}
    \big|d(x_i,\Ymap_{x_iy_j}p)-d(x_i,\Ymap_{x_iy_j}q)\big|
    \le d(\Ymap_{x_iy_j}p,\Ymap_{x_iy_j}q)\le d(p,q),
  \end{equation*}
  and thus by {\cite[Proposition~1.32]{Weaver2}} we get $\|g_{\mu}\|\le 1$.
  Furthermore, for every $k \in \{1,\ldots,n\}$, we have
  \begin{equation*}
    g_{\mu}(x_k)
    \ge
    f(x_k) - \max_{j\in\{1,\ldots,n\}}d(x_k,\Ymap_{x_k y_j}x_k)
    =
    f(x_k) - \max_{j\in\{1,\ldots,n\}}d(x_k,x_k)
    =
    f(x_k)
  \end{equation*}
  and
  \begin{equation*}
    g_{\mu}(y_k)
    \le
    \max_{i\in\{1,\ldots,n\}}
    \big( f(x_i) - d(x_i,\Ymap_{x_i y_k}y_k) \big)
    =
    \max_{i\in\{1,\ldots,n\}}
    \big( f(x_i) - d(x_i,y_k) \big)
    \le
    f(y_k).
  \end{equation*}
  Thus $g_{\mu}(m_{x_ky_k}) \ge f(m_{x_ky_k}) = 1$ for every
  $k \in \{1,\ldots,n\}$ and $g_{\mu}(\mu) = 1$.

  Next we will prove \ref{item2_Lem:function_R_tree}. Fix $u,v\in M$
  and $\alpha>0$ with $u\neq v$ and $m_{uv}\in
  S(g_{\mu},\alpha)$. There exist $i,j\in\{1,\ldots,n\}$ such that
  \begin{equation*}
    g_{\mu}(u)
    =
    f(x_i) - \max_{k\in\{1,\ldots,n\}} d(x_i,\Ymap_{x_iy_k}u)
  \end{equation*}
  and
  \begin{equation*}
    g_{\mu}(v)
    =
    \max_{k\in\{1,\ldots,n\}} \big( f(x_k)-d(x_k,\Ymap_{x_ky_j}v) \big).
  \end{equation*}
  Then
  \begin{align*}
    (1-\alpha)d(u,v)&<g_{\mu}(u) - g_{\mu}(v)
    \le
    f(x_i)-d(x_i,\Ymap_{x_iy_j}u) - f(x_i) + d(x_i,\Ymap_{x_iy_j}v) \\
    &=
    d(x_i,\Ymap_{x_iy_j}v)-d(x_i,\Ymap_{x_iy_j}u)
    \le
    d(\Ymap_{x_iy_j}u, \Ymap_{x_iy_j}v).
    \qedhere
  \end{align*}
\end{proof}

\begin{prop}\label{prop:R_Tree_Delta}
  Let $M$ be a subset of an $\mathbb{R}$-tree $T$ and $\mu\in
  S_{\lipfree{M}}$. Then $\mu$ is a $\Delta$-point if and only if for
  every $\varepsilon>0$ and for every slice $S$ of $B_{\lipfree{M}}$
  with $\mu\in S$ there exist $u,v\in M$ with $u\neq v$ such that
  $m_{uv}\in S$ and $d(u,v)<\varepsilon$.
\end{prop}

\begin{proof}
  One implication can be derived directly from \cite[Theorem~2.6]{JungRueda}.

  Assume that $\mu \in S_{\lipfree{M}}$ is a $\Delta$-point.
  Fix a slice $S(f,\delta)$ with $\mu \in S(f,\delta)$
  and fix $\varepsilon>0$.
  By Lemma~\ref{lem:diminution_slices} we may assume $\delta < 1/2$.
  Let $\gamma>0$ be such that $\mu\in S(f,\delta-\gamma)$.
  There exist $n \in \N$, $\lambda_1,\ldots,\lambda_n>0$
  with $\sum^{n}_{i=1}\lambda_i=1$,
  and $m_{x_1 y_1},\ldots,m_{x_n y_n}\in S_{\lipfree{M}}$
  such that $\sum^{n}_{i=1}\lambda_i m_{x_iy_i}\in S_{\lipfree{M}}$ and
  \begin{equation*}
    \norm{\mu - \sum^{n}_{i=1}\lambda_i m_{x_iy_i}}
    <
    \gamma.
  \end{equation*}
  Let $\nu := \sum^{n}_{i=1}\lambda_i m_{x_iy_i}$ and let $g_\nu\in
  S_{\Lip_0(M)}$ be the function from Lemma \ref{Lem:function_R_tree}.
  Then $g_\nu(\mu) > 1 - \gamma$,
  giving us $(f+g_\nu)(\mu) > 2 - \delta$.

  Since $\mu$ is a $\Delta$-point, then by \cite[Proposition~3.1]{VeeorgFunc}
  there exists a sequence $(m_{u_iv_i})$ such that
  $(f+g_\nu)(m_{u_iv_i}) > 2 - \delta$ and
  \begin{equation*}
    \|m_{u_iv_i} - m_{u_jv_j}\| \ge 2 - \delta
  \end{equation*}
  for all $i,j\in \N$ with $i\neq j$.
  Thus $(m_{u_iv_i}) \subseteq S(f,\delta)$
  and $(m_{u_iv_i})\subseteq S(g_\nu,\delta)$.
  For every $k\in \N$ there exist
  $i,j\in\{1,\ldots,n\}$ such that
  \begin{equation*}
    (1-\delta)d(u_k,v_k)<d(\Ymap_{x_iy_j}v_k, \Ymap_{x_iy_j}u_k).
  \end{equation*}
  By passing to a subsequence of $(m_{u_k v_k})$, if necessary,
  we may assume that the same $x_i$ and $y_j$ work for all $k$.
  Define $x := x_i$ and $y := y_j$.
  The idea is now to use compactness of the segment $[x,y]$ in $T$
  to find $u_k$ and $v_k$ with $d(u_k,v_k)$ arbitrarily small
  by using $\Ymap_{xy}u_k, \Ymap_{xy}v_k \in [x,y]$.

  Note that $\Ymap_{xy}v_k\neq \Ymap_{xy}u_k$,
  since $0<(1-\delta)d(u_k,v_k)<d(\Ymap_{x_iy_j}v_k, \Ymap_{x_iy_j}u_k)$.
  From \eqref{eq:2} we get
  \begin{equation*}
    d(u_k,v_k)
    =
    d(u_k,\Ymap_{xy}u_k) + d(\Ymap_{xy}u_k,\Ymap_{xy}v_k) + d(\Ymap_{xy}v_k,v_k).
  \end{equation*}
  Since $\norm{m_{u_kv_k} - m_{u_lv_l}} \ge 2 - \delta$
  we have, by \cite[Lemma~1.2]{VeeorgStudia},
  \begin{equation*}
    d(u_k,u_l)+d(v_k,v_l)
    \ge
    d(u_k,v_k) + d(u_l,v_l) -
    \delta \max \big\{d(u_k,v_k),d(u_l,v_l)\big\}
  \end{equation*}
  for all $k\neq l$.
  Therefore
  \begin{align*}
    \delta\max \big\{d(u_k,v_k),d(u_l,v_l)\big\}
    &\ge
    d(u_k,v_k) + d(u_l,v_l) - d(u_k,u_l) - d(v_k,v_l)\\
    &\ge
    d(u_k,\Ymap_{xy}u_k) + d(\Ymap_{xy}u_k,\Ymap_{xy}v_k) + d(\Ymap_{xy}v_k,v_k)\\
    &\quad+
    d(u_l,\Ymap_{xy}u_l) + d(\Ymap_{xy}u_l,\Ymap_{xy}v_l) + d(\Ymap_{xy}v_l,v_l)\\
    &\quad-
    \big(
    d(u_k,\Ymap_{xy}u_k) + d(\Ymap_{xy}u_k,\Ymap_{xy}u_l) + d(\Ymap_{xy}u_l,u_l)
    \big)\\
    &\quad-
    \big(
    d(v_k,\Ymap_{xy}v_k) + d(\Ymap_{xy}v_k,\Ymap_{xy}v_l) + d(\Ymap_{xy}v_l,v_l)
    \big)\\
    &=
    d(\Ymap_{xy}u_k,\Ymap_{xy}v_k) + d(\Ymap_{xy}u_l,\Ymap_{xy}v_l)\\
    &\quad-
    d(\Ymap_{xy}u_k,\Ymap_{xy}u_l) - d(\Ymap_{xy}v_k,\Ymap_{xy}v_l)\\
    &>
    (1-\delta) \big(d(u_k,v_k) + d(u_l,v_l)\big)\\
    &\quad
    -
    d(\Ymap_{xy}u_k,\Ymap_{xy}u_l)-d(\Ymap_{xy}v_k,\Ymap_{xy}v_l),
  \end{align*}
  and thus
  \begin{equation*}
    (1-2\delta) \big(d(u_k,v_k) + d(u_l,v_l)\big)
    <
    d(\Ymap_{xy}u_k,\Ymap_{xy}u_l) + d(\Ymap_{xy}v_k,\Ymap_{xy}v_l).
  \end{equation*}
  The segment $[x,y]$ in $T$ is compact and
  therefore there exist $k,l\in I$ such that
  \[
    d(\Ymap_{xy}u_k,\Ymap_{xy}u_l)+d(\Ymap_{xy}v_k,\Ymap_{xy}v_l)
    <(1-2\delta)\varepsilon.
  \]
  Then $d(u_k,v_k)<\varepsilon$ and $m_{u_kv_k}\in S(f,\delta)$ and
  we have found the element we were looking for.
\end{proof}

\begin{thm}\label{thm:R_Tree_Daug_Delta_eq}
  Let $M$ be a subset of an $\mathbb{R}$-tree $T$. Then $\mu\in
  S_{\lipfree{M}}$ is a Daugavet-point if and only if it is a
  $\Delta$-point.
\end{thm}

\begin{proof}
  One direction is trivial, so we assume that
  $\mu \in S_{\lipfree{M}}$ is a $\Delta$-point.
  We may also assume that $M$ is complete.
  Choose $f\in S_{\Lip_0(M)}$ such that $f(\mu)=1$.

  By \cite[Theorem~2.1]{VeeorgStudia} it is enough to show that for
  every $\nu \in \dent(B_{\lipfree{M}})$ we have $\|\mu - \nu\| = 2$.
  So let $\nu$ be a denting point in $B_{\lipfree{M}}$.
  By \cite[Corollary~3.44]{Weaver2} we have $\nu = m_{xy}$ for some
  $x,y \in M$ with $x \neq y$.

  Let $0 < \varepsilon < 1/2$.
  It is enough to show that $\|\mu - m_{xy}\| \ge 2 - 4\varepsilon$.
  If $f(m_{yx}) \ge 1-4\varepsilon$, then
  $\|\mu-m_{xy}\| \ge f(\mu - m_{xy}) \ge 2-4\varepsilon$
  and we are done.

  Now assume $f(m_{yx})<1-4\varepsilon$.
  Next extend $f$ by the McShane--Whitney Theorem to $T$.
  For simplicity assume $x$ is the element $0$.
  Then our assumption is $f(y) < (1 - 4\varepsilon)d(x,y)$.
  Define $g\colon M\rightarrow\mathbb{R}$ by
  \begin{equation*}
    g(p)
    :=
    \min\big\{
    \max\big\{
    d(x,\Ymap_{xy}p) - f(\Ymap_{xy}p) - 2 \varepsilon d(x,y),
    0
    \big\},
    (1-4\varepsilon)d(x,y) - f(y)
    \big\}.
  \end{equation*}
  Note that $g(p) \ge 0$ for all $p \in M$
  and that $g(p) = g(q)$ if $\Ymap_{xy} p = \Ymap_{xy} q$.
  By \cite[Proposition~1.32]{Weaver2} we have
  $g\in \Lip_0(M)$ since, for every $p,q \in M$,
    \begin{equation*}
    |d(x,\Ymap_{xy}p) - f(\Ymap_{xy}p) - d(x,\Ymap_{xy}q) + f(\Ymap_{xy}q)|
    \le 2 d(\Ymap_{xy}p, \Ymap_{xy}q)
    \le 2 d(p,q).
  \end{equation*}
  Note that if $p$ is close to either $x$ or $y$,
  then we have complete control over $g(p)$.
  More specifically, if $p\in B\big(x,\varepsilon d(x,y)\big)$, then
  \begin{equation*}
    d(x,\Ymap_{xy}p)-f(\Ymap_{xy}p)
    \le 2d(x,\Ymap_{xy}p)\le 2\varepsilon d(x,y),
  \end{equation*}
  and thus $g(p)=0$.
  Similarly, if $p\in B\big(y,\varepsilon d(x,y)\big)$, then
  \begin{equation*}
    d(x,\Ymap_{xy}p)-f(\Ymap_{xy}p)
    \ge \big(d(x,y)-d(y,\Ymap_{xy}p)\big)-\big(f(y)+ d(y,\Ymap_{xy}p)\big)
    \ge d(x,y)-f(y)-2\varepsilon d(x,y)
  \end{equation*}
  and therefore $g(p) = (1-4\varepsilon)d(x,y)-f(y) > 0$.

  Since $m_{xy}$ is a denting point, by \cite[Theorem~4.1]{AG19},
  there exists $\delta>0$ such that for $p\in M$
  \begin{equation*}
    p\notin
    B\big(x, \varepsilon d(x,y) \big)
    \cup
    B\big(y, \varepsilon d(x,y) \big)
    \Rightarrow
    d(x,p)+d(y,p) \ge (1+\delta) d(x,y).
  \end{equation*}
  We may assume $\delta < 1 - 2\varepsilon$.

  \begin{claim}{}
    If $g(m_{uv})>0$ for some $u,v \in M$ with $u\neq v$,
    then $d(u,v)\ge \delta d(x,y)/2$.
  \end{claim}
  \begin{claimproof}{}
    First assume $u,v\in B\big(x,\varepsilon d(x,y)\big)
    \cup B\big(y,\varepsilon d(x,y)\big)$.
    From the properties of $g$ it follows that
    $d(y,u) \le \varepsilon d(x,y)$ and
    $d(x,v) \le \varepsilon d(x,y)$, hence
    \begin{equation*}
      d(u,v)
      \ge d(x,y)-2\varepsilon d(x,y)
      \ge \delta d(x,y).
    \end{equation*}
    Next, the cases
    $u \notin B\big(x,\varepsilon d(x,y)\big)
    \cup B\big(y,\varepsilon d(x,y)\big)$
    and
    $v \notin B\big(x,\varepsilon d(x,y)\big)
    \cup B\big(y,\varepsilon d(x,y)\big)$
    are analogous, so we only consider
    $u \notin B\big(x,\varepsilon d(x,y)\big)
    \cup B\big(y,\varepsilon d(x,y)\big)$.
    Then
    \begin{equation*}
      (1+\delta) d(x,y)
      \le
      d(x,u)+d(y,u)
      =
      d(x,y)+2d(\Ymap_{xy}u,u)
    \end{equation*}
    and thus $\delta d(x,y)\le 2d(\Ymap_{xy}u,u)$.
    Furthermore, $g(u) \neq g(v)$ so $\Ymap_{xy}u\neq \Ymap_{xy}v$,
    and thus by \eqref{eq:2}
    \begin{equation*}
      d(u,v)
      =
      d(u,\Ymap_{xy}u) + d(\Ymap_{xy}u,\Ymap_{xy}v) + d(\Ymap_{xy}v,v)
      \ge d(\Ymap_{xy}u,u)
      \ge \delta d(x,y)/2
    \end{equation*}
    and the claim is proved.
  \end{claimproof}

  Now assume that $g(\mu)\neq0$.
  Since $\mu$ and $-\mu$ are both $\Delta$-points,
  by Proposition~\ref{prop:R_Tree_Delta} there exists
  $m_{uv}\in S\big(g/\|g\|,1\big)$ such that $d(u,v)<\delta d(x,y)/2$,
  contradicting the above claim.

  Therefore $g(\mu)=0$. Let $h=f+g$. Then $h(\mu)=1$ and
  \begin{equation*}
    h(m_{yx})
    = \frac{f(y)+g(y)}{d(x,y)}
    = \frac{f(y)+(1-4\varepsilon)d(x,y)-f(y)}{d(x,y)}
    = 1 - 4\varepsilon
  \end{equation*}
  so that $h(\mu)-h(m_{xy}) = 2-4\varepsilon$.
  Therefore it only remains to show that $\|h\|=1$.
  Let $u,v\in M$ with $u\neq v$. If
  $g(u)\le g(v)$, then
  \begin{equation*}
    h(m_{uv})\le f(m_{uv})\le 1.
  \end{equation*}
  If $g(u)> g(v)$, then $\Ymap_{xy}u\neq \Ymap_{xy}v$.
  Furthermore, $g(u)>0$ and $g(v)< (1-4\varepsilon)d(x,y)-f(y)$,
  because $0\le g(p)\le (1-4\varepsilon)d(x,y)-f(y)$
  for all $p\in M$.
  Thus
  \begin{align*}
    g(u)-g(v)
    &=\min\big\{
    d(x,\Ymap_{xy}u)-f(\Ymap_{xy}u)-2\varepsilon d(x,y),
    (1-4\varepsilon)d(x,y) - f(y)
    \big\}\\
    &\quad-
    \max\big\{d(x,\Ymap_{xy}v)-f(\Ymap_{xy}v) - 2\varepsilon d(x,y), 0\big\}\\
    &\le
    d(x,\Ymap_{xy}u) - f(\Ymap_{xy}u) - d(x,\Ymap_{xy}v) + f(\Ymap_{xy}v)\\
    &\le
    d(\Ymap_{xy}u,\Ymap_{xy}v) - f(\Ymap_{xy}u) + f(\Ymap_{xy}v)
  \end{align*}
  and by \eqref{eq:2} we get
  \begin{align*}
    h(u)-h(v)
    &=
    f(u)-f(v)+g(u)-g(v)\\
    &\le
    f(u) - f(\Ymap_{xy}u) + d(\Ymap_{xy}u,\Ymap_{xy}v)  + f(\Ymap_{xy}v) - f(v)\\
    &\le
    d(u,\Ymap_{xy}u) + d(\Ymap_{xy}u,\Ymap_{xy}v) + d(\Ymap_{xy}v,v)\\
    &=
    d(u,v).
  \end{align*}
  Therefore $\|h\| = 1$ and thus
  \begin{equation*}
    \|\mu-m_{xy}\|\ge
    h(\mu)-h(m_{xy}) = 2-4\varepsilon.
  \end{equation*}
  Since $\varepsilon \in (0,1/2)$ was arbitrary,
  we get $\|\mu-m_{xy}\|=2$ as desired.
\end{proof}

For a complete subset $M$ of an $\mathbb{R}$-tree, we can now combine
this result with Proposition~\ref{prop:Daugavet_convex_series} to
obtain a characterization of those Daugavet-points in $\lipfree{M}$
that are convex combinations or convex series of molecules. Also, we
can actually say a bit more for such elements: they are all super
Daugavet-points. To see this, we will need the following lemma.

\begin{lem}\label{lem:R-tree_convex-recombination}
  Let $M$ be a complete subset of an $\mathbb{R}$-tree $T$ and
  $\mu\in S_{\lipfree{M}}$. If $\mu$ can be written as a convex
  combination of molecules $m_{x_iy_i}\in S_{\lipfree{M}}$ with
  $i\in\set{1,\ldots,n}$,
  and if $[x_i,y_i]\subset M$ for every $i$, then $\mu$ can be written
  as a convex combination of molecules $m_{u_jv_j}$ with
  $j\in\{1,\dots,m\}$ such that
  \begin{enumerate}
    \item\label{item:R-tree_convex-recombination_1} $[u_j,v_j]\subset
      M$ for every $j\in \{1,\ldots,m\}$;
    \item\label{item:R-tree_convex-recombination_2} for every $j,k\in
      \{1,\ldots,m\}$ with $j\neq k$ we have either
      $\Ymap_{u_jv_j}p=u_k$ for every $p\in[u_k,v_k]$ or
      $\Ymap_{u_jv_j}p=v_k$ for every $p\in[u_k,v_k]$.
  \end{enumerate}
\end{lem}

\begin{proof}
For each $i$ let $B_i=\Ymap_{x_iy_i}(\set{x_j,y_j\,:\,j=1,\ldots,n})$
and write this set as $B_i=\set{p_1,\ldots,p_l}$ where $p_1=x_i$,
$p_l=y_i$, and $p_k\in [p_{k-1},p_{k+1}]$ for $1<k<l$. Then
$$
m_{x_iy_i} = \sum_{k=1}^{l-1} \frac{d(p_k,p_{k+1})}{d(x_i,y_i)}m_{p_kp_{k+1}}
$$
is a convex sum. Thus we can substitute this in the original
expression of $\mu$ and assimilate identical terms to obtain an
expression of $\mu$ as a convex sum of distinct molecules $m_{u_jv_j}$
where, for each $j$, $u_j$ and $v_j$ are consecutive points of some
$B_i$. In particular, $[u_j,v_j]\subset M$.

Now \ref{item:R-tree_convex-recombination_2} can be derived
from \eqref{eq:3}. Indeed, since every pair $\{u_j,v_j\}$ is composed of consecutive
points of some $B_i$, it follows that
$\Ymap_{u_jv_j}u_k,\Ymap_{u_jv_j}v_k\in\set{u_k,v_k}$ for all
$j,k$. Moreover, if $\Ymap_{u_jv_j}u_k\neq\Ymap_{u_jv_j}v_k$, then
\eqref{eq:2} shows that $u_j,v_j\in [u_k,v_k]$. Again, because $u_k$
and $v_k$ are consecutive, we get $\set{u_j,v_j}=\set{u_k,v_k}$ and
therefore $j=k$.
\end{proof}

\begin{thm}\label{thm:r_tree_conv_comb_mol}
  Let $M$ be a complete subset of an $\mathbb{R}$-tree $T$ and $\mu\in
  S_{\lipfree{M}} $. If $\mu$ can be written as a convex combination
  or as a convex series of molecules $m_{x_iy_i}$ with $i \in I$, for
  some non-empty subset $I$ of $\N$, then the following are equivalent
  \begin{enumerate}
  \item\label{item:rtr_1}
    $\mu$ is a super Daugavet-point;
  \item\label{item:rtr_2}
    $\mu$ is a $\Delta$-point;
  \item\label{item:rtr_3}
    $m_{x_iy_i}$ is a super Daugavet-point for every $i\in I$;
  \item\label{item:rtr_4}
    $m_{x_iy_i}$ is a $\Delta$-point for every $i\in I$;
  \item\label{item:rtr_5}
    $[x_i,y_i]\subset M$ for every $i\in I$.
  \end{enumerate}
\end{thm}

\begin{proof}
  \ref{item:rtr_3} $\Leftrightarrow$ \ref{item:rtr_4}
  $\Leftrightarrow$ \ref{item:rtr_5}
  is Proposition~\ref{prop:Rtree-mol-warmup}, 
  \ref{item:rtr_1} $\Rightarrow$ \ref{item:rtr_2}
  is clear, and \ref{item:rtr_2} $\Rightarrow$ \ref{item:rtr_4} follows by combining
  Theorem~\ref{thm:R_Tree_Daug_Delta_eq} and
  Proposition~\ref{prop:Daugavet_convex_series}.
  
  So let us prove \ref{item:rtr_5} $\Rightarrow$
  \ref{item:rtr_1}. Since the set of all super Daugavet-points in a
    given Banach space is always closed, it is sufficient to prove the
    implication for finite $I$, and furthermore, by
    Lemma~\ref{lem:R-tree_convex-recombination}, we may assume that
    molecules $m_{x_iy_i}$ are such that
    \begin{enumerate}[label=(\alph*)]
      \item  $[x_i,y_i]\subset M$ for every $i\in I$;

      \item for every $i,j\in I$ with $i\neq j$, we have either
        $\Ymap_{x_iy_i}p=x_j$ for every $p\in[x_j,y_j]$ or
        $\Ymap_{x_iy_i}p=y_j$ for every $p\in[x_j,y_j]$.
  \end{enumerate} We will do the proof by induction on $\abs{I}$. If
  $|I|=1$, then $\mu$ is a super Daugavet-point by Proposition
  \ref{prop:Rtree-mol-warmup}. Let $n\in\mathbb{N}$ and assume that
  the assumption holds when $|I|\le n$. Then assume that $|I|=n+1$ and
  fix $i\neq j\in I$. Then either
  $\Ymap_{x_iy_i}x_j=\Ymap_{x_iy_i}y_j=x_i$ or
  $\Ymap_{x_iy_i}x_j=\Ymap_{x_iy_i}y_j=y_i$. We will only consider the
  first case, the second case is analogous.  For simplicity we will
  assume that $x_i$ is the fixed point 0. Let
    \begin{equation*}
        M_1=\set{p\in M: \Ymap_{x_iy_i}p=x_i},
    \end{equation*}
    and $M_2=(M\setminus M_1)\cup\{x_i\}$.
    By e.g. \cite[Proposition~3.9]{Weaver2} we have that
    $\lipfree{M}$ is isometric to
    $\lipfree{M_1}\oplus_1 \lipfree{M_2}$
    since we have $d(p,q)=d(p,x_i)+d(x_i,q)$ for
    $p\in M_1$, $q \in M_2$. Indeed, this equality is obvious
    for $q = x_i$;
    otherwise $x_i=\Ymap_{x_iy_i}p\neq \Ymap_{x_iy_i}q$ and so by
    \eqref{eq:2}
    \begin{equation*}
        d(p,q) = d(p,x_i)+d(x_i,\Ymap_{x_iy_i}q)+d(\Ymap_{x_iy_i}q,q) = d(p,x_i)+d(x_i,q).
    \end{equation*}
    Now notice that $m_{x_iy_i}\in \lipfree{M_2}$ and
    for every $k\in I\setminus\{i\}$ we have
    $\Ymap_{x_iy_i}x_k=\Ymap_{x_iy_i}y_k$ and thus either $m_{x_ky_k}\in
    \lipfree{M_1}$ or $m_{x_ky_k}\in \lipfree{M_2}$.
    Let
    \begin{equation*}
        J=\set{k\in I: m_{x_ky_k}\in \lipfree{M_1}}.
    \end{equation*}
    Note that $m_{x_jy_j}\in \lipfree{M_1}$ and $m_{x_iy_i}\notin
    \lipfree{M_1}$. Therefore $|J|\le n$ and $|I\setminus J|\le n$. By
    assumption both $\sum_{k\in J}\lambda_km_{x_ky_k}/\sum_{k\in
      J}\lambda_k$ and $\sum_{k\in I\setminus
      J}\lambda_km_{x_ky_k}/\sum_{k\in I\setminus J}\lambda_k$ are
    super Daugavet-points in $\lipfree{M_1}$ and $\lipfree{M_2}$
    respectively. Then by \cite[Proposition~3.27]{MPRZ} we get that
    $(\sum_{k\in J}\lambda_km_{x_ky_k},\sum_{k\in I\setminus
      J}\lambda_km_{x_ky_k})$ is a super Daugavet-point and thus $\mu$
    is a super Daugavet-point, which is what we wanted.
\end{proof}

Let us finally point out that we do not know whether arbitrary
Daugavet-points in Lipschitz-free spaces over subsets of $\R$-trees
need always be super Daugavet-points, or whether arbitrary
$\Delta$-points need  always be super $\Delta$-points.

\subsection{\texorpdfstring{$\mathfrak{D}$}{D}-points in
  subspaces of \texorpdfstring{$L_1$}{L1}}
\label{subsec:subL1-reflexive}
In this subsection we study the existence of $\mathfrak{D}$-points
in subspaces of $L_1[0,1]$. The results in this subsection
also hold for $L_1(\mu)$ spaces with the exact same proofs when
$(\Omega, \Sigma, \mu)$ is a probability measure space,
but in keeping with the rest of the section we only study
subspaces of $L_1[0,1]$.

In \cite[Theorem~3.3]{AHLP} it was shown that if $\nu$ is a
$\sigma$-finite measure, then $L_1(\nu)$ has the Daugavet property
if and only if it has property $(\mathfrak{D})$.
As explained in \cite[Section~5.1]{AlbiacKalton} any such
$L_1(\nu)$ is isometrically isomophic to $L_1(\mu)$
where $\mu$ is a probability measure.
The main result of this section,
Theorem~\ref{thm:propD-implies-Daugprop},
extends \cite[Theorem~3.3]{AHLP} to subspaces,
that is, for subspaces of $L_1[0,1]$ the Daugavet property
coincides with property $(\mathfrak{D})$.
On the way to this result we show that no reflexive subspace
of $L_1[0,1]$ contain $\mathfrak{D}$-points.
In particular, they contain neither $\Delta$-points nor Daugavet-points.

Before we proceed let us note that the question of whether a
reflexive Banach space contains $\Delta$-points or not, is
non-trivial. Indeed, in \cite[Theorem~3.1]{AALMPPV}
a (super)reflexive space with a $\Delta$-point was constructed,
and it is proved in \cite[Theorem~4.5]{AALMPPV} that every
infinite dimensional Banach space can be renormed so as 
to admit a $\Delta$-point. Also, a (super)reflexive space 
with a super Daugavet-point was constructed in \cite{HLPV23}.

Let $\mu$ denote the Lebesgue measure on $[0,1]$. Recall that a bounded subset
$\mathcal F \subset L_1[0,1]$ is called \emph{equi-integrable} if
given $\eps > 0,$ there is $\delta > 0$ so that for every Lebesgue measurable
set $A \subset [0,1]$ with $\mu(A) < \delta$ we have
\[
  \sup_{f \in \mathcal F}\int_A|f|\, d\mu < \eps.
\]
Also recall that the unit ball of any reflexive subspace
of $L_1[0,1]$ is equi-integrable
(see e.g. \cite[Theorem~5.2.8]{AlbiacKalton}).

First we have a geometrically obvious lemma.

\begin{lem}\label{lem:f_not_too_small}
  Let $f \in L_1[0,1]$ and $\beta > 0$.
  There exists $n \in \mathbb{N}$ such that
  \begin{equation*}
    \int_D |f| \, d\mu \ge \frac{\beta}{n}
  \end{equation*}
  for all measurable $D \subseteq \supp(f)$ with $\mu(D) \ge \beta$.
\end{lem}

\begin{proof}
  Assume that $f\neq 0$, let $E:=\supp(f)$ and $\alpha:=\mu(E)$, and
  pick $\beta\leq \alpha$.
  For $k \in \mathbb{N}$ define $A_k := \{|f| \ge \frac{1}{k}\}$.
  Then $A_1 \subset A_2 \subset \cdots$
  and $E = \bigcup_{k=1}^\infty A_k$,
  hence $\alpha = \lim_k \mu(A_k)$.
  Choose $m$ such that $\mu(A_{m}) \ge \alpha - \frac{\beta}{2}$.
  For any measurable $D \subseteq E$ with $\mu(D) \ge \beta$
  we have
  \begin{equation*}
    \mu(A_m\setminus D)\leq \mu(E\setminus D)\leq\alpha-\beta,
  \end{equation*}
  so
  \begin{equation*}
    \mu(A_{m} \cap D) \ge \mu(A_m)-\mu(A_m\setminus D)\geq \beta/2.
  \end{equation*}
  Then
  \begin{equation*}
    \int_D |f| \, d\mu
    \geq
    \int_{A_{m}\cap D} |f| \, d\mu
    \ge
    \frac{1}{m}
    \mu(A_{m} \cap D)
    \ge \frac{\beta}{2m},
  \end{equation*}
  and $n:=2m$ is the integer we are looking for.
\end{proof}

Our first main result of this subsection is next.
We find a way of recognizing $\mathfrak{D}$-points
and use equi-integrability to exclude such points
from reflexive subspaces of $L_1[0,1]$.

\begin{prop}
  \label{prop:4}
  Let $X$ be a subspace of $L_1[0,1],$ $f \in S_X,$ and $h \in D(f).$
  If $f$ is a $\mathfrak{D}$-point, then for all $\eps > 0,$
  all $\delta > 0,$ and all $\gamma >0$ there exist
  a measurable set $D$ with $\mu(D) < \delta$ and $g \in S(h, \gamma)$
  such that
  \[
    \int_D |g|\,d\mu > 1 - \eps.
  \]
  In particular, if $X$ is reflexive, then $X$
  does not have $\mathfrak{D}$-points.
\end{prop}

\begin{proof}
  We argue contrapositively.
  Let $f \in S_X$ and $h \in D(f)$.
  We may and will consider $h$ as an element in $L_\infty[0,1]$.
  Assume there exist $\eps > 0$, $\delta > 0$, and
  $\gamma >0$ such that for all $g\in S(h, \gamma),$ and
  all measurable sets $D$ with $\mu(D) < \delta$ we have
  \[
    \int_D |g|\,d\mu
    \le 1 - \eps.
  \]

  Fix $0 < \gamma_0 < \min\{\eps/2, \gamma\}$
  and let $g \in S(h, \gamma_0).$
  Set
  \begin{align*}
    A_1
    & = \{|f - g| = |f| - |g|\},\\
     A_2
    & = \{|f - g| = |g| - |f|\} \setminus A_1,\\
     A_3
    & = \{|f - g| = |g| + |f|\} \setminus (A_1 \cup A_2).\\
  \end{align*}
  Then
  \begin{align*}
    \int_{[0,1]} |f - g|\,d\mu
    & = \int_{A_1} |f| - |g| \,d\mu
      + \int_{A_2} |g| - |f| \,d\mu
      + \int_{A_3} |g| + |f| \,d\mu\\
    & = \|f\| + \|g\| - 2\int_{A_1} |g|\,d\mu - 2\int_{A_2} |f|\,d\mu.
  \end{align*}
  Note that $A_1 \supset \{g=0\}$ and $A_1 \cup A_2 \supset
  \{f=0\}.$

  \textbf{Case 1}: $\mu(A_2 \cap \supp(f)) \ge \delta.$
  Then by Lemma~\ref{lem:f_not_too_small} there exists
  $n \in \mathbb{N}$ such that
  \[
   \int_{A_2} |f|\,d\mu \ge \frac{\delta}{n},
  \]
  hence $\|f - g\| \le 2 - \frac{2\delta}{n}.$

  \textbf{Case 2:} $\mu(A_2 \cap \supp(f)) < \delta.$ Since
  $g \in S(h, \gamma_0)$ we have $\int_{A_3} |g|\,d\mu < \gamma_0$
  (just note that $A_3 \subset \supp(f) \cap \supp(g)$ so on this set we have $fh > 0$
  and $fg < 0$). Hence from this and the assumption we get
  \begin{align*}
    \|f - g\|
    &\le
      \|f\| + \int_{A_2} |g|\,d\mu + \int_{A_3} |g|\,d\mu\\
    &\le 1 + 1 - \eps + \gamma_0
      = 2 - \frac{\eps}{2}.
  \end{align*}
  Since $\eps$ and $\delta$ are fixed, this implies altogether that
  $f$ is not a $\mathfrak{D}$-point.

  Finally, note that if $X$ is reflexive,
  then $B_X$ is equi-integrable and then for any $\varepsilon > 0$
  there exists $\delta > 0$ such that
  \begin{equation*}
    \int_D |g| \, d\mu < \varepsilon
  \end{equation*}
  for all $g \in B_X$ if $\mu(D) < \delta.$ Hence, 
  the starting assumption of our proof is satisfied
  for all $f \in S_X$ and $h \in D(f)$.
\end{proof}

\begin{rem}
  Let $X$ be a subspace of $L_1[0,1].$
  From the proof above it is clear that we do not need
  the full strength of equi-integrability of $B_X$
  to conclude that $X$ does not have $\mathfrak{D}$-points.

  If $B_X$ has the property that there
  exist $1 >\eps > 0$ and $\delta > 0$ such for all measurable sets
  $A$ with $\mu(A) < \delta$ and all $g \in B_X$ we have
  \[
   \int_A |g|\,d\mu \le 1 - \eps,
  \]
  then $X$ does not have $\mathfrak{D}$-points.

  It can be shown that this property is in fact equivalent
  to reflexivity of $X$.
  Hence for every non-reflexive subspace $X$ of $L_1[0,1]$ we have
  that for all $\varepsilon > 0$ and all $\delta > 0$
  there exist a measurable set $A$ with $\mu(A) < \delta$
  and $g \in B_X$ with
  \begin{equation*}
    \int_A |g|\,d\mu > 1 - \varepsilon.
  \end{equation*}
  It follows that if $X$ is a non-reflexive subspace of $L_1[0,1]$
  and $f_1,\ldots,f_n \in S_X$ and $\varepsilon > 0$,
  then there exists $g \in S_X$ such that $\|f_i + g\| \ge 2 - \varepsilon$.
  Thus we get an alternative route to the known result that
  every non-reflexive subspace of $L_1[0,1]$ is octahedral
  \cite[Proposition~3.4]{MR3119340}.
  It was shown in \cite{MR2844454} that octahedrality of separable spaces
  is equivalent to the almost Daugavet property, and it was also shown
  that a Banach space $X$ has the \emph{almost Daugavet property}
  if and only if there exists a norming subspace $Y \subset X^*$
  such that for every $x \in S_X$ we have
  $\sup_{y \in S} \|x - y\| = 2$ for every slice $S$ of $B_X$
  defined by a functional $y^* \in S_Y$.
\end{rem}

Bourgain and Rosenthal proved that there exists a
subspace $E$ of $L_1[0,1]$ with the Schur property and the DLD2P (see
\cite[ Remarks (1) p. 67 and Corollary~2.3]{MR599302}).
In \cite[Theorem~2.5]{MR2051139} Kadets and Werner showed that this
subspace has the Daugavet property. The following result tells us that this is no
coincidence. In fact any subspace of
$L_1[0,1]$ that has property $(\mathfrak{D})$ has the Daugavet property.

\begin{thm}
  \label{thm:propD-implies-Daugprop}
  Let $X$ be a subspace of $L_1[0,1]$.
  The $X$ has property $(\mathfrak D)$ if and only if
  $X$ has the Daugavet property.
\end{thm}

\begin{proof}
  One direction is trivial, so assume that $X$ has property $(\mathfrak{D})$.
  Let $f_0 \in S_X,$ $ h_0 \in S_{X^*},$ $\eps_0 > 0,$ and $\delta_0
  > 0.$ Since $\nu(E) = \int_E |f_0| \, d\mu$ is continuous
  with respect to $\mu$ there exists $\eta > 0$ such that
  \[
    \int_A |f_0|\,d\mu < \delta_0/4
    \quad \mbox{whenever} \quad
    \mu(A) < \eta.
  \]
  By the Bishop-Phelps theorem, we can find $h \in S_{X^*}$
  that attains its norm on $X$ and $0 < \eps < \delta_0/4$ such that
  \[
    S(h, \eps)
    \subset
    S(h_0, \eps_0).
  \]
  Choose $f \in S(h, \eps)$ with $h(f) = 1.$ By assumption $f$ is a
  $\mathfrak{D}$-point. Hence by Proposition~\ref{prop:4} we can find a
  measurable set $D$ with $\mu(D) < \eta$ and $g \in S(h, \eps)$ such
  that $\int_D |g|\,d\mu > 1 - \eps.$ Now we get
  \begin{align*}
    \|f_0 - g\|
    &=
    \int_D|f_0 - g|\,d\mu + \int_{D^c}|f_0 - g|\,d\mu\\
    &\ge \int_D|g|\,d\mu - \int_{D} |f_0|\,d\mu
    + \int_{D^c} |f_0|\,d\mu - \int_{D^c} |g|\,d\mu\\
    &> 1 - \eps - \delta_0/4  + 1 - \delta_0/4 - \eps
    > 2 - \delta_0,
  \end{align*}
  and the conclusion follows.
\end{proof}

\begin{rem}
  {\ }
  \begin{enumerate}

  \item Note that we cannot weaken the assumption
  in Theorem~\ref{thm:propD-implies-Daugprop}
  from property $(\mathfrak{D})$ to ``all slices contain a $\mathfrak{D}$-point''
  and still have the same conclusion. Indeed, $\ell_1$ gives a counterexample.
  As noted in the last paragraph of Section~\ref{sec:preliminaries},
  every element in $S_{\ell_1}$ with infinite support is a
  $\mathfrak{D}$-point, but $\ell_1$ fails to contain
  $\Delta$-points. Since every slice of $B_{\ell_1}$ contains an
  element with infinite support and $\ell_1$ fails the Daugavet
  property, the claim follows.
  
  \item Every subspace of $L_1[0,1]$ with property $(\mathfrak{D})$ has the
  almost Daugavet property. Based on this and
  Theorem~\ref{thm:propD-implies-Daugprop}, one can ask if it is
  true in general that Banach spaces with property $(\mathfrak{D})$
  and the almost Daugavet property, have the Daugavet
  property. However, this is not true, even if we replace property
  $(\mathfrak{D})$ with the relative Daugavet property.

  Indeed, use $X = L_1[0,1]$ in the construction in
  Corollary~\ref{cor:relative-Daugavet_no_daugavet-point}.
  Then we get a separable Banach space $Y$ which has
  the almost Daugavet property by \cite[Theorem~3.2]{MR3745582}
  (since $L_1[0,1]$ is octahedral),
  and $Y$ fails to contain Daugavet-points.

  \end{enumerate}
  
\end{rem}

\subsection{A subspace of \texorpdfstring{$L_1$}{L1} with
  \texorpdfstring{$\Delta$}{Delta}-points, but no relative
  Daugavet-points}
\label{subsec:subL1_relative-Daugavet_no_daugavet}

We now address the question whether $\Delta$- and Daugavet-points are the
same in all subspaces of $L_1[0,1]$. We already know that the answer is yes
for all subspaces that are isometrically isomorphic to Lipschitz free
spaces over subsets of $\mathbb{R}$-trees and for reflexive subspaces. Nevertheless,
we will construct a subspace of $L_1[0,1]$ with $\Delta$-points, but no
Daugavet-points, in fact we will show more, it contains no relative Daugavet-points.  
We start by introducing some notation.

Let $T = \bigcup_{n \ge 1} \{0,1\}^n$
be the infinite unrooted binary tree.
If $t = (t_1,\ldots,t_n) \in T$, then
for $m \le |t|$ the element $t|m = (t_1,\ldots,t_m)$
is an initial segment of $t$ with length $m$.
A node $t \in T$ is the predecessor of two elements
$t^\frown(0)$ and $t^\frown(1)$.
Even though we have removed the usual root $\emptyset$ from $T$, it
will be convenient to start the following inductive process from
there, and we will thus also consider
$t^\frown(0) = (0)$ and $t^\frown(1) = (1)$ when
$t = \emptyset$.

For $n \in \mathbb{N}$ we define the sets
\begin{equation*}
  B_n = B_n^\emptyset :=
  \left[ \frac{1}{2^{n+1}}, \frac{1}{2^n} \right)
  \quad \mbox{and} \quad
  C_n = C_n^\emptyset :=
  \left\{ \frac{1}{2} \right\} + B_n.
\end{equation*}
If for $t \in T$ or $t = \emptyset$ an interval $A^t = [a,b)$ has
been defined we define
\begin{equation*}
  A^{t^\frown(0)} = \left[ a, \frac{a+b}{2} \right)
  \quad \mbox{and} \quad
  A^{t^\frown(1)} = \left[ \frac{a+b}{2}, b \right).
\end{equation*}
Using this we have defined the sets
$B_n^t$ and $C_n^t$ for every $n \in \mathbb{N}$
and every $t \in T$.
Let us note some obvious but useful properties of these sets.
\begin{fact}\label{fact:Bt}
  The followings holds for all $t \in T$ and $n \in \N$.
  \begin{enumerate}
  \item\label{item:Btprop1}
    $\mu(B_n) = \mu(C_n) = 2^{-n-1}$;
  \item\label{item:Btprop2}
    $\mu(B_n^t) = \mu(C_n^t) = 2^{-|t|}\mu(B_n)
    = 2^{-|t|}\mu(C_n) = 2^{-|t|-n-1}$;
  \item\label{item:Btprop3}
    $\bigcup_{|t|=n} B_n^t = B_n$ and
    $\bigcup_{|t|=n} C_n^t = C_n$;
  \item\label{item:Btprop4}
    $\mu(C_n \cup \bigcup_{i=1}^n B_i) = \frac{1}{2}$
    for every $n \in \mathbb{N}$.
  \end{enumerate}
\end{fact}

\begin{defn}\label{defn:f_t_h_t}
  For each $t \in T$ we define the following functions in $L_1[0,1]$
  \begin{equation*}
    f_t := 2^{|t|+1}
    \left(
      \1_{C_{|t|}^t} + \sum_{i=1}^{|t|} \1_{B_i^t}
    \right)
    \quad
    \mbox{and}
    \quad
    h_t := 2^{|t|+1} \sum_{i=1}^\infty \1_{B_i^t}.
  \end{equation*}
\end{defn}

Figure~\ref{fig:functions} shows a picture of
$f_{(0)}$,
$f_{(1)}$,
$f_{(0,0)}$ and
$f_{(1,1)}$.

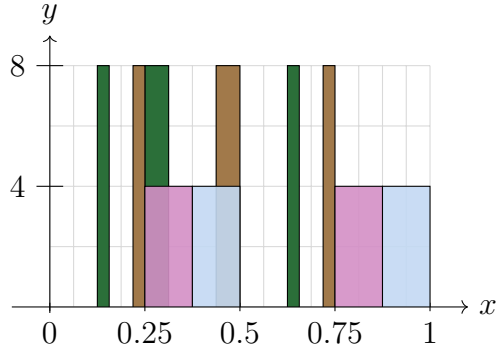
\begin{figure}[h]
\begin{center}
  \begin{tikzpicture}[yscale=0.4,xscale=5]
    \draw[help lines,gray!30!white,ystep=2,xstep=0.0625] (0,0) grid (1,8);

    \foreach \x in {0,0.25,0.5,0.75,1} {
      \draw (\x,4pt) -- (\x,-4pt) node[below] {$\x$};
    }
    \foreach \y in {4,8} {
        \draw (1pt,\y) -- (-1pt,\y) node[left] {$\y$};
    }

    \draw[black,fill=fill1]
    (0.125,0) -- (0.125,8) -- (0.1562,8) -- (0.1562,0);
    \draw[black,fill=fill1]
    (0.25,0) -- (0.25,8) -- (0.3125,8) -- (0.3125,0);
    \draw[black,fill=fill1]
    (0.625,0) -- (0.625,8) -- (0.6562,8) -- (0.6562,0);
    \draw[black,fill=fill2]
    (0.2188,0) -- (0.2188,8) -- (0.25,8) -- (0.25,0);
    \draw[black,fill=fill2]
    (0.4375,0) -- (0.4375,8) -- (0.5,8) -- (0.5,0);
    \draw[black,fill=fill2]
    (0.7188,0) -- (0.7188,8) -- (0.75,8) -- (0.75,0);
    \draw[black,fill=fill3,fill opacity=0.9]
    (0.25,0) -- (0.25,4) -- (0.375,4) -- (0.375,0);
    \draw[black,fill=fill3,fill opacity=0.9]
    (0.75,0) -- (0.75,4) -- (0.875,4) -- (0.875,0);
    \draw[black,fill=fill4,fill opacity=0.9]
    (0.375,0) -- (0.375,4) -- (0.5,4) -- (0.5,0);
    \draw[black,fill=fill4,fill opacity=0.9]
    (0.875,0) -- (0.875,4) -- (1,4) -- (1,0);

    \draw[->] (-0.1,0) -- (1.1,0) node[right] {$x$};
    \draw[->] (0,-0.2) -- (0,9) node[above] {$y$};
  \end{tikzpicture}
\end{center}
\caption{The functions \colorbox{fill3}{$f_{(0)}$},
  \colorbox{fill4}{$f_{(1)}$},
  \colorbox{fill1}{$f_{(0,0)}$}, and
  \colorbox{fill2}{$f_{(1,1)}$}.}
\label{fig:functions}
\end{figure}

With the main actors introduced we are ready to state
the main theorem of this subsection. 

\begin{thm}\label{thm:subsp_delta_not_relative-Daugavet}
  Define $X := \cspn(f_t)_{t\in T}$ and $Y := \cspn(h_t)_{t\in T}$.
  Then $X$ is a non-reflexive subspace of $L_1[0,1]$
  without the DLD2P,
  and $Y$ is a subspace of $X$ isometrically isomorphic to $L_1[0,1]$.

  Every $y \in S_Y$ is a $\Delta$-point in $Y$ and $X$,
  but none of these points are relative Daugavet-points in $X$.
  Furthermore, these points are super $\Delta$-points in $X$.
\end{thm}

\begin{rem}
  Actually, one can show that every element of $X$ whose support
  is contained in $[0,1/2)$ belongs to $Y$ and that no
  function whose support intersects $[1/2,1]$ can be a $\Delta$-point.
  So the $\Delta$-points in $X$ are precisely the
  elements on the unit sphere of $Y$, and
  $X$ contains no relative Daugavet-point.
\end{rem}

Throughout this subsection $X$ and $Y$ will denote the
two subspaces of $L_1[0,1]$ from the above theorem.
We will show that the $f_t$'s are strongly exposed points of $B_X$,
so not all elements in $S_X$ are $\Delta$-points.
But there is no lack of $\Delta$-points in $X$
since $Y$ is isometric to $L_1[0,1]$.
In fact, the $h_t$'s form a 1-separated dyadic tree in $L_1[0,1]$.

For the proof of Theorem~\ref{thm:subsp_delta_not_relative-Daugavet}
we will need a series of technical lemmas.
We begin with some properties of $f_t$ and $h_t$.

\begin{lem}\label{lem:prop_of_ht}
  For each $t \in T$ we have $\|f_t\| = \|h_t\| = 1$ and $h_t \in X$.
  In particular, $Y$ is a subspace of $X$.
\end{lem}

\begin{proof}
  Let $t \in T$. Then
  \begin{equation*}
    \|f_t\|
    = 2^{|t|+1}
    \biggl(
    \mu(C_{|t|}^t) + \sum_{i=1}^{|t|} \mu(B_i^t)
    \biggr)
    = 2^{|t|+1}
    \frac{1}{2^{|t|}}
    \mu\biggl(C_{|t|} \cup \bigcup_{i=1}^{|t|} B_i\biggr)
    =
    1
  \end{equation*}
  by Fact~\ref{fact:Bt}~\ref{item:Btprop2} and \ref{item:Btprop4} above.
  Fact~\ref{fact:Bt}~\ref{item:Btprop2} also gives
  \begin{equation*}
    \|h_t\| = 2^{n+1}\mu\biggl(\bigcup_{i=1}^\infty B_i^t\biggr)
    = 2^{n+1} \frac{1}{2^n} \mu([0,1/2))
    = 1.
  \end{equation*}

  To show that $Y$ is a subspace of $X$
  it is enough to show that each $h_t \in X$.
  Let $t \in T$ with $|t| = n$.
  Define
  \begin{equation*}
    S_t^m = \{
    u \in T : u|n = t, |u| = |t| + m
    \}
  \end{equation*}
  so that $S_t^1 = S_t$ is the set of immediate successors of $t$.
  Note that $|S_t^m| = 2^m$.
  Define $h_t^m \in \spn(f_t)_{t \in T}$ by
  \begin{align*}
    h_t^m
    &:=
    \frac{1}{2^m} \sum_{u \in S_t^m} f_u
    =
    \frac{1}{2^m} \sum_{u \in S_t^m} 2^{m+n+1}
    \left(
      \1_{C_{|u|}^u} + \sum_{i=1}^{|u|} \1_{B_i^u}
    \right) \\
    &=
    2^{|t|+1}
    \left(
      \1_{C_{m+n}^t} + \1_{\cup_{i=1}^{m+n} B_i^t}
    \right).
  \end{align*}
  It should be reasonably clear that
  this sequence is norm convergent to $h_t$, i.e. 
  \begin{equation*}
    h_t = \lim_m h_t^m
    = 2^{|t|+1} \sum_{i=1}^\infty \1_{B_i^t} \in X
  \end{equation*}
  as desired.
\end{proof}

We could have noted in Lemma~\ref{lem:prop_of_ht}
that the $h_t$'s can be approximated by convex
combinations of $f_u$'s, with $|u|$ arbitrarily large,
and since the measure of the support of the $f_u$'s tend
to zero we have that every $h_t$ is a $\Delta$-point.
But there are many more as the Daugavet property
of $L_1[0,1]$ and the following lemma shows.

\begin{lem}\label{lem:subspace_isom_L1}
  The subspace $Y$ of $X$ is
  isometrically isomorphic to $L_1[0,1]$.
\end{lem}

\begin{proof}
  It is known (see e.g. \cite[p.~211]{MR737004}) that a Banach space
  is isometrically isomorphic to $L_1[0,1]$ if and only if it is the
  closed linear span of a system $(g_t)_{t \in T}$ such that
  \begin{enumerate}
  \item
    For each $t \in T$ we have
    \begin{equation*}
      g_t = \frac{1}{2}(g_{t^\frown (0)} + g_{t^\frown (1)});
    \end{equation*}
  \item
    For each $n \in \N$ each $2^{n}$-tuple $(a_t)_{|t|=n}$ we have
    \begin{equation*}
      \sum_{|t|=n} |a_t|
      =
      \norm{ \sum_{|t|=n} a_t g_t }.
    \end{equation*}
  \end{enumerate}
  For $(h_t)_{t \in T}$ we have
  \begin{equation*}
    h_t = 2^{|t|+1} \sum_{i=1}^\infty \1_{B_i^t}
    =
    \frac{1}{2}\left(
      2^{|t|+2} \sum_{i=1}^\infty (\1_{B_i^{t^\frown(0)}} + \1_{B_i^{t^\frown(1)}})
    \right)
    =
    \frac{1}{2}\left(h_{t^\frown(0)} + h_{t^\frown(1)}\right)
  \end{equation*}
  and
  \begin{equation*}
    \norm{ \sum_{|t|=n} a_t h_t }
    =
    \sum_{|t|=n} |a_t| \|h_t\|
    =
    \sum_{|t|=n} |a_t|
  \end{equation*}
  by using Lemma~\ref{lem:prop_of_ht} and the fact that for each
  $i \in \N$ we have $B_{i}^t \cap B_{i}^s = \emptyset$ when
  $s \neq t$ and $|s|=|t|$.
\end{proof}

\begin{rem}
  In the above proof the idea is to map the
  system $(g_t)_{t \in T}$ to the characteristic
  functions of the dyadic intervals whose span
  is dense in $L_1[0,1]$
  (see e.g. \cite[Proposition~6.1.3]{AlbiacKalton}).
  This is an idea we will meet again
  in the proof of Theorem~\ref{thm:subsp_delta_not_relative-Daugavet}.
\end{rem}

The next aim is to show that the $f_t$'s are strongly exposed.
Before we can get there we need a few technical lemmas that
focus on the norm of elements in the span of the $f_t$'s.
We will show that there is a limit on how much of the norm
can be concentrated on the first few $B_i$'s.

\begin{lem}\label{lem:norm_expr_for_span_ft}
  Let 
  \begin{equation*}
    g := \sum_{t \in T} \alpha_t f_t \in \linsp(f_t)_{t\in T}
  \end{equation*}
  be such that $\alpha_t = 0$ for all $t \in T$ with $|t| > n$.
  Then for any $t \in T$ with $|t| = n$ and any $1 \le j \le n$
  we have
  \begin{equation*}
    \|\restr{g}{B_j^t}\|
    =
    2^{-n}
    \left|
      \sum_{i=j}^n 2^{i-j} \alpha_{t|i}
    \right|
  \end{equation*}
  and
  \begin{equation*}
    \|\restr{g}{C_j^t}\| = 2^{-n} |\alpha_{t|j}|.
  \end{equation*}
  In particular,
  \begin{equation*}
    \|g\|
    =
    \sum_{|t| = n}
    \sum_{j=1}^n
    \left( \|\restr{g}{B_j^t}\| + \|\restr{g}{C_j^t}\| \right)
    =
    2^{-n}
    \sum_{|t| = n}
    \sum_{j=1}^n
    \left(
      \left|
        \sum_{i=j}^n 2^{i-j} \alpha_{t|i}
      \right|
      +
      |\alpha_{t|j}|
    \right).
  \end{equation*}
\end{lem}

\begin{rem}
  With the notation from the above lemma we have
  \begin{equation*}
    \|\restr{g}{B_j}\|
    = \sum_{|t|=n} \|\restr{g}{B_j^t}\|
    = 2^{-n}\sum_{|t|=n}
    \left|
      \sum_{i=j} 2^{i-j} \alpha_{t|i}
    \right|
  \end{equation*}
  and
  \begin{equation*}
    \|\restr{g}{C_j}\|
    = \sum_{|t|=n} 2^{-n} |\alpha_{t|j}|
    = 2^{n-j} \sum_{|s|=j} 2^{-n} |\alpha_{s}|
    = 2^{-j} \sum_{|s|=j} |\alpha_{s}|.
  \end{equation*}
\end{rem}

\begin{proof}
  Let
  \begin{equation*}
    g := \sum_{t \in T} \alpha_t f_t \in \linsp(f_t)_{t\in T}
  \end{equation*}
  be such that $\alpha_t = 0$ for all $t \in T$ with $|t| > n$.

  Let $t \in T$ with $|t| = n$.
  Then by definition of $f_t$ and Fact~\ref{fact:Bt}~\ref{item:Btprop2}
  \begin{align*}
    \| \restr{g}{B_1^t} \|
    &=
    \left\|
      \sum_{i=1}^n \alpha_{t|i} \restr{f_{t|i}}{B_1^t}
    \right\|
    =
    \left\|
      \sum_{i=1}^n 2^{i+1} \alpha_{t|i} \1_{B_1^t}
    \right\| \\
    &=
    2^{-|t| - 1 - 1}
    \left|
      \sum_{i=1}^n 2^{i+1} \alpha_{t|i}
    \right|
    =
    2^{-n}
    \left|
      \sum_{i=1}^n 2^{i-1} \alpha_{t|i}
    \right|.
  \end{align*}
  Since $\supp(f_s) \cap B_2 = \emptyset$ when $|s|=1$,
  we analogously get
  \begin{equation*}
    \| \restr{g}{B_2^t} \|
    =
    \left\|
      \sum_{i=2}^n 2^{i+1} \alpha_{t|i} \1_{B_2^t}
    \right\|
    =
    2^{-|t| - 2 - 1}
    \left|
      \sum_{i=2}^n 2^{i+1} \alpha_{t|i}
    \right|
    =
    2^{-n}
    \left|
      \sum_{i=2}^n 2^{i-2} \alpha_{t|i}
    \right|
  \end{equation*}
  Similar reasoning gives
  \begin{equation*}
    \|\restr{g}{B_j^t}\|
    =
    2^{-n}
    \left|
      \sum_{i=j}^n 2^{i-j} \alpha_{t|i}
    \right|
  \end{equation*}
  for $1 \le j \le n$.

  The support of the $f_t$'s on the $C_j$'s are disjoint
  so for $t \in T$ with $|t| = n$ (max level of $g$)
  we have
  \begin{align*}
    \left\| \restr{g}{C_j^t} \right\|
    &=
    \left\|
      \alpha_{t|j} \restr{f_{t|j}}{C_j^t}
    \right\|
    =
    \left\|
      2^{j+1} \alpha_{t|j} \1_{C_j^t}
    \right\|
    =
    2^{j+1} |\alpha_{t|j}| \mu(C_j^t) \\
    &=
    2^{j+1} |\alpha_{t|j}| 2^{-|t|-j-1}
    =
    2^{-n} |\alpha_{t|j}|.
  \end{align*}
  The last part of the statement now follows by
  summing over disjoint parts of $\supp(g)$.
\end{proof}

The proof of the next lemma is trivial so we skip it.

\begin{lem}\label{lem:trivial_norm_obs}
  Let $X$ be a Banach space and let $\varepsilon > 0$.
  If $x,y \in X$ and $z = x + y$ satisfy
  \begin{equation*}
    \|x\| + \|y\| = \|z\|,
  \end{equation*}
  then the following are equivalent
  \begin{enumerate}
  \item\label{item:norm1}
    $\|x\| \le (1-\varepsilon)\|z\|$;
  \item\label{item:norm2}
    $\|y\| \ge \varepsilon \|z\|$;
  \item\label{item:norm3}
    $\|x\| \le (\varepsilon^{-1} - 1)\|y\|$.
  \end{enumerate}
\end{lem}

\begin{lem}\label{lem:norm_sum_ineq}
  For any sequence of scalars $(\alpha_i)$
  and any $m, n \in \mathbb{N}$ with $m \le n$
  the following inequality holds:
  \begin{equation*}
    \left|\sum_{i=m}^{n} 2^{i-m}\alpha_i\right|
    \leq
    \sum_{j=m+1}^{n} \left|\sum_{i=j}^{n} 2^{i-j}\alpha_i\right|
    + \sum_{i=m}^{n} \left|\alpha_i \right|.
  \end{equation*}
  In particular $\|\restr{g}{B_{k}}\| \leq \left\|
  \restr{g}{\bigcup_{j=k+1}^n B_j \cup \bigcup_{j=1}^n C_j}
  \right\|$.
\end{lem}

\begin{proof}
  First we note that the case $m > 1$ follows
  from the case $m = 1$ by relabeling.
  So we will assume as we may that $m = 1$, and show that
  \begin{equation*}
    \left|\sum_{i=1}^{n} 2^{i-1}\alpha_i\right|
    \leq
    \sum_{j=2}^{n} \left|\sum_{i=j}^{n} 2^{i-j}\alpha_i\right|
    + \sum_{i=1}^{n} \left|\alpha_i \right|
  \end{equation*}
  by induction on $n$.
  For $n = 1$ the statement is trivial.
  Now suppose the inequality holds for some $n = k$.
  Then for $n = k+1$ we have
  \begin{align*}
    \left|\sum_{i=1}^{k+1}2^{i-1}\alpha_i\right|
    & =
    \left|
      \sum_{i=2}^{k+1}(2^{i-2}\alpha_i + 2^{i-2}\alpha_i) + \alpha_1
    \right|
    \\
    & \leq
    \left|
      \sum_{i=2}^{k+1}2^{i-2}\alpha_i
    \right|
    +
    \left|
      \sum_{i=2}^{k+1}2^{i-2}\alpha_i
    \right|
    +
    \left|
      \alpha_1
    \right|
    \\
    &\le
    \left(
      \sum_{j=3}^{k+1}
      \left|
        \sum_{i=j}^{k+1}2^{i-j}\alpha_{i}
      \right|
      +
      \sum_{j=2}^{k+1} \left|\alpha_{i} \right|
    \right)
    +
    \left|\sum_{i=2}^{k+1}2^{i-2}\alpha_i\right|
    +
    \left| \alpha_1   \right|
    \\
    &= \sum_{j=2}^{k+1}
    \left|
      \sum_{i=j}^{k+1} 2^{i-j}\alpha_i
    \right|
    + \sum_{i=1}^{k+1} \left|\alpha_i \right|,
  \end{align*}
  where the first inequality is the triangle inequality
  and the second is the induction hypothesis.

  The particular case follows directly as
  \begin{align*}
         \|\restr{g}{B_{k}}\|
         &=
         2^{-n} \sum_{|t|=n}
         \left|
         \sum_{i=k}^n 2^{i-k} \alpha_{t|i}
         \right|
         \\
         &\le
         2^{-n} \sum_{|t|=n}
         \left(
         \sum_{j=k+1}^n
         \left|
         \sum_{i=j}^n 2^{i-j} \alpha_{t|i}
         \right|
         +
         \sum_{i=k}^n |\alpha_{t|i}|
         \right)
         \\
         &=
         \sum_{j=k+1}^n \|\restr{g}{B_j}\|
         +
         \sum_{j=k}^n \|\restr{g}{C_j}\|
         \le
         \left\|
         \restr{g}{\bigcup_{j=k+1}^n B_j \cup \bigcup_{j=1}^n C_j}
         \right\|. \qedhere
  \end{align*}
\end{proof}

\begin{lem}\label{lem:concentraded_norm2}
  Let
  \begin{equation*}
    g := \sum_{t \in T} \alpha_t f_t \in \linsp(f_t)_{t\in T}
  \end{equation*}
  be such that $\alpha_t = 0$ for all $t \in T$ with $|t| > n$.
  Then for any $1 \le m \le n$
  we have
  \begin{equation*}
    \left\|\restr{g}{\bigcup_{i=1}^m B_i}\right\| \leq (1-2^{-m})\left\|g\right\|.
  \end{equation*}
\end{lem}

\begin{proof}
  From the implication \ref{item:norm3} to \ref{item:norm1}
  in Lemma~\ref{lem:trivial_norm_obs} with $\varepsilon = 2^{-m}$
  we see that it is enough to show
  \[
    \left\|\restr {g}{\bigcup_{i=1}^m B_i}\right\|
    \leq
    (2^{m}-1)
    \left\|
      \restr{g}{\bigcup_{i=m+1}^nB_{i}\cup \bigcup_{i=1}^{n} C_{i}}
    \right\|.
  \]

  We show this by induction. For $m = 1$ this
  is just the particular case of Lemma~\ref{lem:norm_sum_ineq}.

  Assume the hypothesis holds for $m= k < n$. That is, assume
  \begin{align*}
    \left\|\restr {g}{\bigcup_{i=1}^k B_i}\right\|
    \leq
    (2^{k}-1)
    \left\|
      \restr{g}{\bigcup_{i=k+1}^n B_{i}\cup \bigcup_{i=1}^{n} C_{i}}
    \right\|.
  \end{align*}
  Now we have by first using the induction hypothesis
  and then Lemma~\ref{lem:norm_sum_ineq}
  \begin{align*}
    \left\|\restr {g}{\bigcup_{i=1}^{k+1} B_i}\right\|
    &=
    \left\|\restr {g}{\bigcup_{i=1}^{k} B_i}\right\|
    +
    \left\|\restr {g}{B_{k+1}}\right\|
    \\
    &\le
    (2^{k}-1)
    \left\|
      \restr{g}{\bigcup_{i=k+1}^n B_{i}\cup \bigcup_{i=1}^{n} C_{i}}
    \right\|
    +
    \left\|\restr {g}{B_{k+1}}\right\|
    \\
    &=
    (2^{k}-1)
    \left\|
      \restr{g}{\bigcup_{i=k+2}^n B_{i}\cup \bigcup_{i=1}^{n} C_{i}}
    \right\|
    +
    2^k \left\|\restr {g}{B_{k+1}}\right\|
    \\
    &\le
    (2^{k}-1)
    \left\|
      \restr{g}{\bigcup_{i=k+2}^n B_{i}\cup \bigcup_{i=1}^{n} C_{i}}
    \right\|
    +
    2^k
    \left\|
      \restr{g}{\bigcup_{i=k+2}^n B_{i}\cup \bigcup_{i=1}^{n} C_{i}}
    \right\|
    \\
    &\le
    (2^{k+1}-1)
    \left\|
      \restr{g}{\bigcup_{i=k+2}^n B_{i}\cup \bigcup_{i=1}^{n} C_{i}}
    \right\|
  \end{align*}
  so we are done.
\end{proof}

\begin{lem}\label{lem:ft-denting}
  The functions $f_t$ are strongly exposed points of $B_X$.
\end{lem}

\begin{proof}
  Fix $\eps>0$, and let $t \in T$, $\delta = 2^{-|t|}\varepsilon$,
  and $x^* = \1_{\supp(f_t)} \in S_{X^*}$.
  We will show that the diameter of the slice $S(x^*,\delta)$ goes to
  $0$ as $\delta\to0$.

  Let $h \in S(x^*,\delta)$, and let us assume as we may that
  $h := \sum_{u \in T} \alpha_u f_u \in \linsp(f_u)_{u\in T}$.
  Let $n$ be such that $\alpha_u = 0$ for all $u$
  with $|u| > n$. We may assume that $m := |t| \le n$.
  Define $g := h - \alpha_t f_t$.

  We have
  \begin{equation*}
    x^*(h) \le \|h\| - \|\restr{h}{\supp(f_t)^c}\|
    \le 1 - \|\restr{g}{\supp(f_t)^c}\|.
  \end{equation*}
  From Lemma~\ref{lem:concentraded_norm2} we get
  \begin{equation*}
    \left\|\restr{g}{\supp(f_t)}\right\|
    =
    \left\|\restr{g}{\bigcup_{i=1}^m B_i^{t}}\right\|
    \le
    \left\|\restr{g}{\bigcup_{i=1}^m B_i}\right\|
    \le
    (1 - 2^{-|t|})\|g\|.
  \end{equation*}
  Hence, by Lemma~\ref{lem:trivial_norm_obs}
  with $\varepsilon = 2^{-|t|}$, we get
  \begin{equation*}
    \left\|\restr{g}{\supp(f_t)^c}\right\|
    \ge
    2^{-|t|} \|g\|.
  \end{equation*}
  Now
  \begin{equation*}
    1 - 2^{-|t|}\varepsilon = 1 - \delta
    < x^*(h) \le 1 - 2^{-|t|}\|g\|
  \end{equation*}
  so that $\|g\| < \varepsilon$.
  As $\|h\| = \|\alpha_t f_t + g\| \le 1$ and
  \begin{equation*}
    \alpha_t - \varepsilon
    <
    \alpha_t - \|g\|
    \le
    \|\alpha_t f_t + g\|
    \le 1
    \le \alpha_t + \|g\| < \alpha_t + \varepsilon
  \end{equation*}
  we conclude that $|1-\alpha_t| < \varepsilon$.

  This implies that
  \begin{equation*}
    \left\|f_t - h\right\|
    \leq
    \left\| (1-\alpha_t)f_t \right\|
    +
    \|g \|
    = |1- \alpha_t| + \varepsilon
    < 2\varepsilon,
  \end{equation*}
  and this shows that the diameter of the slice
  tends to $0$ as $\delta \to 0$.
\end{proof}

We have now collected everything we need in order
to show the main theorem.

\begin{proof}[Proof of Theorem~\ref{thm:subsp_delta_not_relative-Daugavet}]
  By Lemma~\ref{lem:prop_of_ht} $Y$ is a subspace of $X$
  and by Lemma~\ref{lem:subspace_isom_L1} the space
  $Y$ is isometric to $L_1[0,1]$.
  Hence $Y$ has the Daugavet property and every
  $g \in S_Y$ is a super $\Delta$-point in $Y$.
  Such points are also super $\Delta$-points
  in every superspace, in particular in $X$.

  Since $f_t$, $t \in T$, is strongly exposed in $B_X$
  by Lemma~\ref{lem:ft-denting} not every element
  of $S_X$ is a $\Delta$-point.
  Hence $X$ does not have the DLD2P.

  It only remains to show that no $y \in S_Y$
  can be a relative Daugavet-point in $X$.
  Let $g \in S_Y$, $x^* \in S_{X^*}$ and $\varepsilon > 0$.
  Assume $x^*(g) = 1$.
  We may assume that $x^*$ considered as a function in $L_\infty[0,1]$
  satisfies $\supp(x^*) \subseteq \supp(g)$.

  Find $g_n \in \linsp(h_t)_{t\in T}$ with $\|g_n\| = 1$
  such that $\|g - g_n\| < \varepsilon$ and
  \begin{equation*}
    g_n = \sum_{|t| = n} \alpha_t h_t
    = \sum_{|t| = n} |\alpha_t| \sign(\alpha_t) h_t
  \end{equation*}
  for some $n \in \mathbb{N}$.
  (This corresponds to taking conditional expectation
  of the $L_1$ function with respect to the $\sigma$-algebra where
  $[0,1]$ is divided into $2^n$ equal parts.)

  We have by using disjointness that
  \begin{equation*}
    1 = \|g_n\| = \sum_{|t|=n} |\alpha_t|\|h_t\|
    = \sum_{|t|=n} |\alpha_t|.
  \end{equation*}
  This means that $g_n \in \conv(\pm h_t)$ and
  hence there exists $t$ with $|t|=n$ such that
  $\sign(\alpha_t)h_t \in S(x^*,\varepsilon)$
  since $g_n \in S(x^*,\varepsilon)$.

  With $h_t^m$ as in Lemma~\ref{lem:prop_of_ht}
  we may find $m$ such that
  \begin{equation*}
    \sign(\alpha_t) h_t^m
    =
    \frac{1}{2^m} \sum_{u \in S_t^m} \sign(\alpha_t) f_u
    \in S(x^*,\varepsilon).
  \end{equation*}
  But this implies that there exists $u \in T$ such that
  $\sign(\alpha_t)f_u \in S(x^*,\varepsilon)$.

  By the assumption on $x^*$ we have
  $\supp(g) \cap \supp(f_u) \neq \emptyset$
  and hence
  \begin{equation*}
    \max_{\pm} \mu(\{ |g \pm f_u| < |g| + |f_u| \}) > 0.
  \end{equation*}
  Therefore $\min_{\pm} \|g \pm f_u\| < 2$.
  By Proposition~\ref{prop:distance-to-denting_relative-Daugavet} we get that
  $g$ is not a relative Daugavet-point since $f_u$ is strongly exposed
  and in particular denting.
\end{proof}

\begin{rem}
  The above proof shows that $B_Y \subset \clco(\pm f_t)$.
  This is not the case for all $f \in B_X$.

  Indeed, let $t \in T$ and define
  \begin{equation*}
    f = f_{t^\frown(0)} + f_{t^\frown(1)} - 2f_t.
  \end{equation*}
  Then
  \begin{equation*}
    \supp(f) = B_{|t| + 1}^t \cup C_{|t|+1}^t  \cup C_{|t|}^t.
  \end{equation*}
  Here $f$ is positive on $P := B_{|t| + 1}^t \cup C_{|t|+1}^t$
  and negative on $N := C_{|t|}^t$.
  Let $g = f/\|f\|$ and $x^* = \1_{P} - \1_{N}$.
  Then $x^*(g) = 1$ by construction.

  For $s \in T$ we have a few cases.

  \begin{enumerate}
  \item
    If $|s| \le |t|$ and $s \neq t$, then $x^*(f_s) = 0$.
  \item
    If $s = t$, then $x^*(f_t) = -2^{|t|+1}\mu(C_{|t|}^t) = 2^{-|t|}$.
  \item
    If $|s| > |t|$ and $t$ is not a predecessor of $s$, 
    then $x^*(f_s) = 0$.
  \item
    If $|s| > |t|$ and $t$ is a predecessor of $s$, 
    then $\supp(f_s) \cap \supp(f) \subseteq B_{|t|+1}^t$ and
    \begin{align*}
      x^*(f_s)
      &=
      \int_{B_{|t|+1}^t} 2^{|s|+1}\1_{B_{|t|+1}^s} \, d\mu
      =
      2^{|s|+1}\mu(B_{|t|+1}^s)
      \\
      &=
      2^{|s|+1}2^{-|s|-|t|-1-1}
      =
      2^{-|t|-1}.
    \end{align*}
  \end{enumerate}
  This implies that $\clco(\pm f_t)$ can be separated
  from $\{f\}$ by $x^*$.
\end{rem}

\section*{Acknowledgments}
\label{sec:ack}

This work was supported by: \begin{itemize}

  \item the AURORA mobility programme (project
  number: 309597) from the Norwegian Research Council and the ``PHC
  Aurora'' program (project number: 45391PF), funded by the French
  Ministry for Europe and Foreign Affairs, the French Ministry for
  Higher Education, Research and Innovation;

  \item the Grant BEST/2021/080 funded by the Generalitat Valenciana, Spain, 
  and by Grant PID2021-122126NB-C33 
  funded by MCIN/AEI/10.13039/501100011033 and by 
  ``ERDF A way of making Europe'';

  \item the Estonian Research Council grants PRG 1901 and SJD58;

  \item the French ANR project No. ANR-20-CE40-0006.
  
  \end{itemize}

Parts of this research were conducted while
T. A. Abrahamsen, R. J. Aliaga and V. Lima visited
the Laboratoire de Math\'ematiques de Besan\c{c}on in 2021, 
while R. J. Aliaga visited the
Institute of Mathematics and Statistics at the University of Tartu 
in 2022, and while T. Veeorg visited the
Department of Mathematics at the University of Agder
in 2022 and the Laboratoire de Math\'ematiques de Besan\c{c}on in 2023, for which they wish to express their gratitude.

\newcommand{\etalchar}[1]{$^{#1}$}
\providecommand{\bysame}{\leavevmode\hbox to3em{\hrulefill}\thinspace}
\providecommand{\MR}{\relax\ifhmode\unskip\space\fi MR }
\providecommand{\MRhref}[2]{%
  \href{http://www.ams.org/mathscinet-getitem?mr=#1}{#2}
}
\providecommand{\href}[2]{#2}

\end{document}